\documentclass{article}
\usepackage{graphicx}
\usepackage{amsfonts}
\usepackage{amsmath}
\usepackage{amssymb}
\usepackage{url}
\usepackage{fancyhdr}
\usepackage{indentfirst}
\usepackage{enumerate}
\usepackage{dsfont}
\usepackage[british]{babel}
\usepackage[colorlinks=true,citecolor=blue]{hyperref}
\usepackage{amsthm}
\usepackage{color}
\usepackage{natbib}
\setcitestyle{numbers,square}
\renewcommand{\cite}{\citet}

\usepackage{comment}
\usepackage{authblk}

\def\d{\mathrm{d}}
\def\laweq{\buildrel \mathrm{d} \over =}

\newcommand{\E}{\mathbb{E}}

\newcommand{\R}{\mathbb{R}}

\newcommand{\p}{\mathbb{P}}

\newcommand{\J}{\mathcal{J}}

\newcommand{\id}{\mathds{1}}
\newcommand{\lcx}{\preceq_{\mathrm{cx}}}

\newcommand{\lv}{\preceq_{\mathrm{h}}}
\newcommand{\lh}{\preceq_{\mathrm{h}}}

\renewcommand{\ge}{\geqslant}
\renewcommand{\le}{\leqslant}
\renewcommand{\geq}{\geqslant}
\renewcommand{\leq}{\leqslant}
\renewcommand{\epsilon}{\varepsilon}

\renewcommand{\cdots}{\dots}

\theoremstyle{plain}
\newtheorem{theorem}{Theorem}
\newtheorem{cor}[theorem]{Corollary}
\newtheorem{lem}[theorem]{Lemma}
\newtheorem{prop}[theorem]{Proposition}
\theoremstyle{definition}
\newtheorem{defn}[theorem]{Definition}
\newtheorem{exam}[theorem]{Example}
\newtheorem{rem}[theorem]{Remark}
\numberwithin{equation}{section} \numberwithin{theorem}{section}

\usepackage{tikz-qtree}

\begin{document}

\title{Distributional Compatibility for Change of Measures}

\author[1]{Jie Shen}
\author[1]{Yi Shen}
\author[2]{Bin Wang}
\author[1]{Ruodu Wang}
\affil[1]{\footnotesize Department of Statistics and Actuarial Science, University of Waterloo, Canada \normalsize}
\affil[2]{\footnotesize RCSDS, NCMIS, Academy of Mathematics and Systems Science, Chinese Academy of Sciences, China}


 \date{\today}
\maketitle

\begin{abstract}

In this paper, we characterize compatibility of distributions and probability measures on a measurable space.
For a set of indices $\mathcal J$, we say that the tuples of probability measures $(Q_i)_{i\in \J} $ and distributions $(F_i)_{i\in \J} $ are {compatible} if there exists a random variable  having distribution $F_i$ under $Q_i$ for each $i\in \mathcal J$.
We first establish an equivalent condition using conditional expectations for general (possibly uncountable) $\J$.
For a finite $n$, it turns out that compatibility of $(Q_1,\dots,Q_n)$ and $(F_1,\dots,F_n)$ depends on the heterogeneity among $Q_1,\dots,Q_n$ compared with that among $F_1,\dots,F_n$.
We show that, under an assumption that the measurable space is rich enough, $(Q_1,\dots,Q_n)$ and $(F_1,\dots,F_n)$
are compatible if and only if $(Q_1,\dots,Q_n)$ dominates $(F_1,\dots,F_n)$ in a notion of heterogeneity order,  defined via multivariate convex order between the Radon-Nikodym derivatives  of $(Q_1,\dots,Q_n)$ and $(F_1,\dots,F_n)$ with respect to  some  reference measures.
  We then proceed to generalize our results to stochastic processes, and conclude the paper with an application to portfolio selection problems under multiple constraints. 



\textbf{Keywords}:  change of measure, compatibility, heterogeneity order, optimization.
   \end{abstract}


\section{Introduction}

\subsection{The main problem}
Change of probability measures is  found ubiquitous in problems where multiple probability measures appear, with extensive theoretical treatment and applications in the fields of probability theory,  statistics, decision theory, simulation, and finance.

A key feature of a change of measure is that the distribution of a random variable is transformed to another one, and this serves many theoretical as well as practical purposes, such as in the modification of a Brownian motion drift (e.g.~\cite{RY13}) or in importance sampling (e.g.~\cite{S76,GL05}).
In view of this, a  question   seems natural to us: \emph{how much} would the distribution change?
We formulate this question below.\\

(A) Given two probability measures $P$ and $Q$ defined on the same measurable space $(\Omega,\mathcal A)$, suppose that a random variable $X:\Omega\to \R$ has a given distribution function $F$ under $P$.
 What are  the possible distributions of $X$ under $Q$?\\

Question (A) arises naturally if one has statistical (distributional) information about a random variable $X$ under $P$, but yet she is concerned about the behaviour of $X$ under another measure $Q$.
A general version of question (A), the vocal focus of this paper, is the following.\\

(B) Given several probability measures $Q_1,\dots,Q_n$ defined on $(\Omega,\mathcal A)$, and distribution measures $F_1,\dots,F_n$ on $\R$,
does there exist a random variable $X:\Omega\to \R$ such that $X$ has distribution $F_i$ under $Q_i$ for $i=1,\dots,n$?

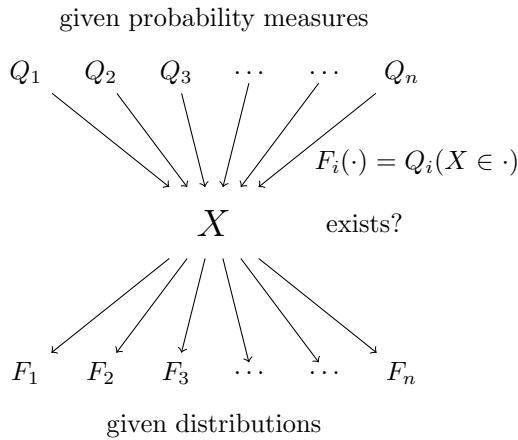
 \tikzstyle{bag} = [text width=2em, text centered]
\tikzstyle{end} = []
\begin{center}
\begin{tikzpicture}[sloped]
        \node (a1) at ( 0,0) [bag] {$Q_1$};
                \node (a2) at ( 1,0) [bag] {$Q_2$};
                        \node (a3) at ( 2,0) [bag] {$Q_3$};
                                \node (a4) at ( 3,0) [bag] {$\cdots$};
                                        \node (a5) at ( 4,0) [bag] {$\cdots$};
                                                \node (a6) at ( 5,0) [bag] {$Q_n$};
                                \node (b) at ( 2.5,-2) [text width=3em, text height=1.5em, text centered, text depth=0.5em] {\Large$ {X}$};
                                \node (d) at ( 4.5,-2) [text width=5em, text centered] {exists?};
                                        \node (e) at ( 2.5,0.7) [text width=12em, text centered] { given probability measures};
                                                \node (f) at ( 2.5,-4.7) [text width=12em, text centered] {given distributions};
                                                                                        \node (g) at (5.2,-1.2) [text width=8em, text centered] { $F_i(\cdot)=Q_i(X\in \cdot)$};
                                    \node (c1) at ( 0,-4) [bag] {$F_1$};
                \node (c2) at ( 1,-4) [bag] {$F_2$};
                        \node (c3) at ( 2,-4) [bag] {$F_3$};
                                \node (c4) at ( 3,-4) [bag] {$\cdots$};
                                        \node (c5) at ( 4,-4) [bag] {$\cdots$};
                                                \node (c6) at ( 5,-4) [bag] {$F_n$};
        \draw  [->] (a1) to (b);
                \draw  [->] (a2) to (b);
                        \draw  [->] (a3) to (b);
                                \draw  [->] (a4) to (b);
                                        \draw  [->] (a5) to (b);
                                                \draw  [->] (a6) to (b);
                                    \draw  [->] (b) to (c1);
           \draw  [->] (b) to (c2);
            \draw  [->] (b) to (c3);
            \draw  [->] (b) to (c4);
            \draw  [->] (b) to (c5);
            \draw  [->] (b) to (c6);
            \end{tikzpicture}
\end{center}

Question (B) is henceforth referred to as the \emph{compatibility} problem  for the $n$-tuples of measures $(Q_1,\dots,Q_n)$ and $(F_1,\dots,F_n)$.
We give an analytical answer to question (B), and hence (A).
More generally, we also address the compatibility of two infinite collections of measures.

Before describing our findings, let us look at a few intuitive cases of (B).
 Suppose that $(Q_1,\dots,Q_n)$ and $(F_1,\dots,F_n)$ are compatible, that is, (B) has an affirmative answer.
 In case that $Q_1,\dots,Q_n$ are identical, it is clear that the  respective distributions  of a random variable under each $Q_i$, $i=1,\dots,n$ are the same; thus $F_1=\dots=F_n$.  In case that  $Q_1,\dots,Q_n$ are mutually singular, the respective distributions of a random variable under $Q_i$, $i=1,\dots,n$ can be arbitrary.
In case that $F_1,\dots,F_n$ are mutually singular measures on $(\R,\mathcal B(\R))$, $Q_1,\dots,Q_n$ have to be also mutually singular.
From the above observations, it then seems natural to us that whether $(Q_1,\dots,Q_n)$ and $(F_1,\dots,F_n)$ are compatible depends on the \emph{heterogeneity} (in some sense) among $Q_1,\dots,Q_n$ compared to that of $F_1,\dots,F_n$. More precisely,  $Q_1,\dots,Q_n$ need to be \emph{more heterogeneous} than $F_1,\dots,F_n$ to allow for compatibility.

To describe the above heterogeneity mathematically, we seek help from a notion of heterogeneity order. It turns out that compatibility of $(Q_1,\dots,Q_n)$ and $(F_1,\dots,F_n)$ is closely related to multivariate convex order between the Radon-Nikodym derivatives
$ (\frac{\d F_1}{\d F},\dots, \frac{\d F_n}{\d F} )$ and   $(\frac{\d Q_1}{\d Q},\dots,\frac{\d Q_n}{\d Q}) $, where $F$ and $Q$ are two ``reference probability measures" on $(\R,\mathcal B(\R))$ and $(\Omega,\mathcal A)$, respectively. In particular, we show that
question (B) has an affirmative answer only if for some measures $F$ dominating $(F_1,\dots,F_n)$ and $Q$  dominating $(Q_1,\dots,Q_n)$,
\begin{equation*}
\int_\R f\left(\frac{\d F_1}{\d F},\dots, \frac{\d F_n}{\d F} \right)\d F \le \int_\Omega f\left(\frac{\d Q_1}{\d Q},\dots,\frac{\d Q_n}{\d Q}\right) \d Q
\end{equation*}
for all convex functions $f:\R^n\to \R$.
Furthermore, if the measurable space $(\Omega,\mathcal A)$ is rich enough, the above necessary condition is sufficient for a positive answer to (B).
We then proceed to generalize our results to stochastic processes, and conclude the paper with  an optimization problem  related to compatibility of distributions under change of measures.



\subsection{Relation  to finance and economics}

The main objective of this paper, question (B), has several deep connections to fundamental problems in finance and economics. We summarize some notable relevant points below\footnote{We thank Marcel Nutz for suggesting the second and the third connections, and Fabio Maccheroni for helpful discussions leading to the fourth connection.}.

\textbf{Risk assessment under multiple scenarios.} In   the evaluation of capital requirement for market risks,  one often needs to assess risk models under different probability measures, e.g.~stressed and non-stressed scenarios.
The evaluation of a risk would then be a combination of distributions obtained under various scenarios.
For a theoretical treatment of this approach and its relation to the Fundamental Review of the Trading Book, we refer to \cite{WZ18}.
A natural question in this context is  whether one can find a risk model (represented by a random variable or a stochastic process) that has  specified distributions under corresponding scenarios.
For instance, one may be interested in simulating from a risk model which has a specific dynamic under a non-stressed scenario and
 another dynamic (e.g.~with different parameters) under a stressed scenario.
The existence of such a risk model is precisely question (B); see Section \ref{sec:girsa} for results on Brownian motions.
For some other questions in the same spirit,  we refer to \cite{EHW16, EMS02} where the authors address a few questions on the existence of certain models satisfying given constraints, which are  raised by practitioners from the financial industry.
In Section \ref{sec:5}, we present a portfolio selection problem with constraints under multiple scenarios.

 \textbf{Simultaneous mass transport.} By definition,  question (B) is equivalent to the existence of a Monge mass transport from $Q_i$ to $F_i$ for all $i=1,\dots,n$ {simultaneously}.
Optimal mass transport is an active topic with various applications in mathematical finance, in particular in the calculation of model-independent bounds;
we refer to  \cite{H11}, \cite{BHP13,BNT17} and \cite{BJ16} for recent advances.
In the study of optimal transport, one typically  looks at an optimal transport for one pair of measures.
The existence of such a transport is trivial for one pair of measures, which corresponds to question (B) for $n=1$.
In this paper, we deal with the case $n>1$, and thus \emph{simultaneous mass transport}. Existence is no longer a trivial issue, and it has to be studied before one could discuss optimality.
Admittedly, we are not aware of immediate applications of simultaneous mass transport in finance. Nevertheless, this paper serves as a starting point for future studies in this direction.

 \textbf{Task assignment problem.} The third connection is a classic task assignment problem in economics.
Suppose that $\Omega$ represents a finite set of workers.
Each worker $\omega\in \Omega$ has some resources of several skills, represented by real numbers $p_1(\omega),\dots,p_n(\omega)$.
One needs to assign workers to stations (represented by real numbers) where each station demands each skill at a specified amount, and each amount integrates to the corresponding total available resources $\int_\Omega p_i(\omega) \d \omega $, $i=1,\dots,n$.
The problem is whether there exists a way to assign the workers so that each station has exactly the amount of skills it demands.
The continuous version of this problem is precisely question (B), where $p_i$ represents the density of the probability measure $Q_i$, $i=1,\dots,n$.

 \textbf{Consequentialism in decision theory.} The fourth connection is found in decision theory.
An Anscombe-Aumann   act (\cite{AA63}) is a vector of distributions, resulting from a lottery (a random variable) under a set of beliefs (a collection of probability measures $Q_1,\dots,Q_n$).
Question (B) is equivalent to the existence of a given Anscombe-Aumann act for a pre-specified set of beliefs.
Decision theorists often study preferences over the set of all Anscombe-Aumann acts without specifying the measures $Q_1,\dots,Q_n$, and this is referred to as an axiom of \emph{consequentialism} (see e.g.~\cite{BCMM17}).
Such an approach assumes that any choice of an act always exists, which is guaranteed by assuming the mutual singularity of $(Q_1,\dots,Q_n)$
(Proposition \ref{prop:properties} (iv) and Theorem \ref{th:meas.variable}). However, mutual singularity is not the case for many parametric models of beliefs.
As such, the results in our paper are helpful to a better understanding of the decision-theoretical framework of consequentialism.

\subsection{Notation}

Throughout, we work with a fixed measurable space $(\Omega, \mathcal A)$, which allows for atomless probability measures.
A probability measure $Q$ on $(\Omega,\mathcal A)$ is said to be atomless if for all $A\in \mathcal A$ with $Q(A)>0$, there exists $B\in \mathcal A$, $B\subset A$ such that $0<Q(B)<Q(A)$. Equivalently, there exists a random variable in $(\Omega,\mathcal A)$ that is continuously distributed under $Q$.
Let $\mathcal F$ be the set of probability measures on $(\R,\mathcal B(\R))$, where $\mathcal B(\R)$ stands for the Borel $\sigma$-algebra of $\R$, and $\mathcal M_1=\mathcal M_1(\Omega)$ be the set of probability measures on $(\Omega, \mathcal A)$.
Let  $\mathcal L^0(\Omega; \mathcal Y)$ be the set of all measurable functions from $\Omega$ to $\mathcal Y$, when the corresponding $\sigma$-fields are clear.
For any measures $Q,Q_1,\dots,Q_n$, we say that $Q$ dominates $(Q_1,\dots,Q_n)$, denoted by $(Q_1,\dots,Q_n)\ll Q$,  if $Q$ dominates $Q_i$ for each $i=1,\dots,n$.

 \section{Compatibility and an equivalent condition}\label{sec:sec2}

We first define  the main concept of this paper,  compatibility problem for the two groups of measures $(Q_i)_{i\in\mathcal J}\subseteq \mathcal M_1$ and $(F_i)_{i\in\mathcal J}\subseteq \mathcal F$, where $\mathcal J$ is a possibly infinite set of indices.
\begin{defn}\label{def1}
$(Q_i)_{i\in\mathcal J}\subseteq \mathcal M_1$ and $(F_i)_{i\in\mathcal J}\subseteq \mathcal F$ are \emph{compatible} if there exists a random variable $X$ in $(\Omega,\mathcal A)$ such that
 $ F_i$ is the distribution of $X$ under $Q_i$  for each $i\in\mathcal J$.
\end{defn}

We note that  $ F $ is the distribution of $X$ under $Q$ if and only if $F=Q\circ X^{-1}$.
Below we establish our first result, which leads to an equivalent condition for compatibility of $(Q_i)_{i\in\mathcal J}\subseteq \mathcal M_1$ and $(F_i)_{i\in\mathcal J}\subseteq \mathcal F$.

\begin{theorem}\label{th:newstuff} For $(Q_i)_{i\in\mathcal J}\subseteq \mathcal M_1$, $(F_i)_{i\in\mathcal J}\subseteq \mathcal F$ and $X\in \mathcal L^0(\Omega;\mathbb R)$, assuming
 that  there exists a probability measure in $\mathcal M_1$ dominating $(Q_i)_{i\in \mathcal J} $,
 equivalent are:
\begin{enumerate}[(i)]
\item   $X$ has distribution $F_i$ under $Q_i$ for $i\in \mathcal J$.
\item  For all $Q\in \mathcal M_1$ dominating $(Q_i)_{i\in\mathcal J}$,  the probability measure $F=Q\circ X^{-1}$ dominates $(F_i)_{i\in\mathcal J}$, and for all $i\in\mathcal J$,
 \begin{equation}\label{eq:eq-new1}  \frac{\d F_i}{\d F}(X) =   \E^Q\left[\left.\frac{\d Q_i}{\d {Q}}\right|X\right].\end{equation}
 \item For some  $Q\in \mathcal M_1$ dominating $(Q_i)_{i\in\mathcal J}$,  the probability measure $F=Q\circ X^{-1}$ dominates $(F_i)_{i\in\mathcal J}$, and (\ref{eq:eq-new1}) holds.
 \end{enumerate}
\end{theorem}
\begin{proof}

 (i)$\Rightarrow$(ii):  By definition, $X$ is such that $Q_i(X\in A)=F_i(A)$ for $A\in \mathcal B(\R)$ and $i\in\mathcal J$. Let $Q\in\mathcal M_1$ such that $Q_i\ll Q$, $i\in\mathcal J$. 
        For any $A\in \mathcal B(\mathbb R)$, if $F(A)=0$, then $Q(X\in A)=0$. Since $Q_i\ll Q$, $Q_i(X\in A)=F_i(A)=0$, we have $F_i\ll F$ for $i\in\mathcal J$.
 We can verify that for any $A\in \mathcal B(\R)$ and $i\in\mathcal J$,
\begin{align*}\E^{Q}\left[\id_{\{X\in A\}}\frac{\d Q_i}{\d Q}\right]&=Q_i(X\in A)\\
&=F_i(A) =\int_A \frac{\d F_i}{\d F}\d F =\E^{Q}\left[\id_{\{X\in A\}}\frac{\d F_i}{\d F}(X)\right].
 \end{align*}
Therefore, $$\frac{\d F_i}{\d F}(X)= \E^{Q}\left[\left.\frac{\d Q_i}{\d Q}\right|X\right],~~i\in \mathcal J.$$

(ii)$\Rightarrow$(iii): Trivial.

(iii)$\Rightarrow$(i): Suppose that \eqref{eq:eq-new1} holds and $F$ dominates $(F_i)_{i\in \mathcal J}$. One can easily verify that, for all $A\in\mathcal B(\R)$ and $i\in\mathcal J$,
\begin{align*}
\E^{Q_i}[\id_{\{X\in A\}}]=\E^{Q}\left[\id_{\{X\in A\}}\frac{\d Q_i}{\d Q}\right]&=\E^Q\left[\E^{Q}\left[\left.\id_{\{X\in A\}}\frac{\d Q_i}{\d Q}\right|X\right]\right]
\\&=\E^Q\left[\id_{\{X\in A\}}\E^{Q}\left[\left.\frac{\d Q_i}{\d Q}\right|X\right]\right]
\\&=\E^Q\left[\id_{\{X\in A\}} \frac{\d F_i}{\d F} (X)\right] =F_i(A).
\end{align*}
Therefore, $X$ has distribution $F_i$ under $Q_i$, $i\in\mathcal J$, thus $(Q_i)_{i\in\mathcal J}$ and $(F_i)_{i\in\mathcal J}$ are compatible.
\end{proof}

\begin{rem}\label{rem:21}
In the case where the index set $\mathcal J=\{1,\dots, n\}$ is finite, a probability measure $Q\in \mathcal M_1$ dominating $(Q_1,\dots, Q_n)$ always exists, as we can take, for example, $Q=\frac{1}{n}(Q_1+\cdots+Q_n)$. As such, the existence assumption in Theorem \ref{th:newstuff} can be removed when $\mathcal J$ is finite.
\end{rem}

From Theorem \ref{th:newstuff}, the necessary and sufficient condition of compatibility is the existence of $X\in \mathcal L^0(\Omega;\mathbb R)$ satisfying \eqref{eq:eq-new1} for some $Q\in \mathcal M_1$ dominating $(Q_i)_{i\in\mathcal J}$. This condition is not easy to verify in general. In the next sections we explore necessary and sufficient conditions, much easier to verify, based on  distributional properties of the random vectors $(\frac{\d F_1}{\d F},\dots, \frac{\d F_n}{\d F} )$ and $(\frac{\d Q_1}{\d {Q}},\dots,\frac{\d Q_n}{\d {Q}})$, where $F$ and $Q$ are some measures dominating $(F_1,\dots,F_n)$ and $(Q_1,\dots,Q_n)$ respectively.
%

\begin{rem}\label{cor:cor2}
In the special case of $n=2$ and  $Q_1\ll Q_2$, one can take $Q=Q_2$ in Theorem \ref{th:newstuff}, and the two-dimensional equality in \eqref{eq:eq-new1} reduces to a one-dimensional equality
$$ \frac{\d F_1}{\d F_2}(X) =   \E^{Q_2}\left[\left.\frac{\d Q_1}{\d {Q_2}}\right|X\right].$$
\end{rem}

\section{Characterizing compatibility via heterogeneity order}\label{sec:sec3}

In this section, we explore analytical conditions for compatibility of $(Q_1,\dots,Q_n)$ and $(F_1,\dots,F_n)$ based on their Radon-Nikodym derivatives with respect to some reference probability measures, which are much easier to verify than Theorem \ref{th:newstuff}.

 \subsection{Preliminaries on convex order}

For an arbitrary probability space $(\Gamma, \mathcal S,P)$, denote by $\mathcal L^1(\Gamma; \mathbb R^n)$  the set of all integrable $n$-dimensional random vectors defined on $(\Gamma, \mathcal S,P)$.
Multivariate convex order is a natural notion of heterogeneity order, as  defined below.

\begin{defn}[Convex order]
Let $(\Omega_1,\mathcal A_1, P_1)$ and $(\Omega_2,\mathcal A_2,P_2)$ be two  probability spaces.
For $\mathbf  X\in\mathcal L^1(\Omega_1; \mathbb R^n)$ and  $\mathbf Y\in \mathcal L^1(\Omega_2; \mathbb R^n)$,
we write $\mathbf X|_{P_1}\preceq_{\rm cx} \mathbf Y|_{P_2}$, if $\E^{P_1}[f(\mathbf X)]\le \E^{P_2}[f(\mathbf Y)]$ for all convex functions $f:\R^n\to \R$.
\end{defn}

For more on multi-dimensional convex order, we refer to \citet[Chapter 3]{MS02} and \citet[Chapter 7]{SS07}.


For $\mathbf  X\in \mathcal L^1(\Omega_1; \mathbb R^n)$ and  $\mathbf Y\in \mathcal L^1(\Omega_2; \mathbb R^n)$, we use $\mathbf X|_{P_1}\laweq \mathbf Y|_{P_2}$ to represent that $\mathbf X$ and $\mathbf Y$ have the same distribution under $P_1$ and $P_2$ respectively.
Clearly, if $\mathbf X|_{P_1}\laweq \mathbf Y|_{P_2}$, then  $\mathbf X|_{P_1}\preceq_{\rm cx} \mathbf Y|_{P_2}$ and  $\mathbf Y|_{P_2}\preceq_{\rm cx} \mathbf X|_{P_1}$.
A key feature of convex order is its connection to conditional expectations. Below in Lemma \ref{lem:lem2} we quote Theorem 7.A.1 of \cite{SS07} for this well-known result (an extension of Strassen's theorem, \cite{S65}); one also finds a slightly simpler formulation as Theorem 3.4.2 of \cite{MS02}. See also \cite{HPRY11} for a construction similar to Lemma  \ref{lem:lem2} for stochastic processes (termed \emph{peacocks}).
\begin{lem} \label{lem:lem2}
For $\mathbf  X\in \mathcal L^1(\Omega_1; \mathbb R^n)$ and  $\mathbf Y\in \mathcal L^1(\Omega_2; \mathbb R^n)$,  $\mathbf X|_{P_1}\preceq_{\rm cx} \mathbf Y|_{P_2}$
if and only if there exist a probability space $(\Omega_3,\mathcal A_3, P_3)$ and $\mathbf X',\mathbf Y'\in \mathcal L^1(\Omega_3; \mathbb R^n)$  such that
$\mathbf X'|_{P_3}\laweq \mathbf X|_{P_1}$, $\mathbf Y'|_{P_3}\laweq \mathbf Y|_{P_2}$,
and $\E^{P_3}[\mathbf Y'|\mathbf X']=\mathbf X'$.
\end{lem}

\subsection{Heterogeneity order}\label{sec:sec32}

As mentioned in the introduction, compatibility intuitively concerns the heterogeneity among $(Q_1,\dots,Q_n)$ compared to $(F_1,\dots,F_n)$.
The following lemma, based on Theorem \ref{th:newstuff}, yields a possible way of characterizing the comparison between the two tuples of measures.
More precisely, a necessary condition for compatibility is built on a convex order relation between the random vectors $(\frac{\d F_1}{\d F},\dots, \frac{\d F_n}{\d F} )$ and $(\frac{\d Q_1}{\d {Q}},\dots,\frac{\d Q_n}{\d {Q}})$ for some reference probability measures $F\in\mathcal F$ and $Q\in \mathcal M_1$.


 \begin{lem}\label{lem}
  If $(Q_1,\dots,Q_n)\in \mathcal M_1^n$ and $(F_1,\dots,F_n)\in \mathcal F^n$ are compatible, then
for any $Q\in \mathcal M_1$ dominating $(Q_1,\dots,Q_n)$, there exists $F\in \mathcal F$ dominating $(F_1,\dots,F_n)$, such that
 \begin{equation}\label{eq:eq1}
\left.\left(\frac{\d F_1}{\d F},\dots, \frac{\d F_n}{\d F} \right)\right|_F \lcx \left. \left(\frac{\d Q_1}{\d Q},\dots,\frac{\d Q_n}{\d Q}\right)\right|_Q.
\end{equation}
Moreover,  $F$ in \eqref{eq:eq1} can be taken as $Q\circ X^{-1}$, where $X$ is a random variable with distribution $F_i$ under $Q_i$, $i=1,\dots,n$.

\end{lem}

\begin{proof}
This lemma is directly obtained from Theorem \ref{th:newstuff} and Lemma \ref{lem:lem2}. More precisely, by Theorem \ref{th:newstuff}, there exists $X\in \mathcal L^0(\Omega;\mathbb R)$ such that  $$  \left(\frac{\d F_1}{\d F},\dots, \frac{\d F_n}{\d F} \right)(X) =   \E^Q\left[\left.\left(\frac{\d Q_1}{\d {Q}},\dots,\frac{\d Q_n}{\d {Q}}\right)\right|X\right]$$  where $F=Q\circ X^{-1}$.
Therefore,
$$
\left.\left(\frac{\d F_1}{\d F},\dots, \frac{\d F_n}{\d F} \right)\right|_{F} \laweq \left. \E^Q\left[\left.\left(\frac{\d Q_1}{\d {Q}},\dots,\frac{\d Q_n}{\d {Q}}\right)\right|X\right]\right|_{Q} \lcx \left. \left(\frac{\d Q_1}{\d Q},\dots,\frac{\d Q_n}{\d {Q}}\right)\right|_{Q},
$$
where the last inequality follows from Lemma \ref{lem:lem2} by taking
$$
\mathbf X'=\E^Q\left[\left.\left(\frac{\d Q_1}{\d {Q}},\dots,\frac{\d Q_n}{\d {Q}}\right)\right|X\right], \mathbf Y'=\left(\frac{\d Q_1}{\d Q},\dots,\frac{\d Q_n}{\d {Q}}\right), P_3=Q.
$$
\end{proof}


 We summarize  the necessary condition in Lemma \ref{lem} for compatibility   by introducing the following heterogeneity order, which is shown to be a partial order in Lemma \ref{lem:lem33} below.
In the following, $\mathcal M_1(\Omega_1)$ and $\mathcal M_1(\Omega_2)$ represent the sets of probability measures on two arbitrary measurable spaces $\Omega_1$ and $\Omega_2$, respectively.
 \begin{defn}\label{def2}
\emph{$(P_1,\dots,P_n)\in \mathcal M_1^n(\Omega_1)$ is dominated by $(Q_1,\dots,Q_n)$ $\in \mathcal M_1^n(\Omega_2)$ in heterogeneity}, denoted by $(P_1,\dots,P_n)\lv  (Q_1,\dots,Q_n)$, if
\begin{equation}\label{eq:question}
\left.\left(\frac{\d P_1}{\d P},\dots, \frac{\d P_n}{\d P} \right)\right|_P \lcx \left. \left(\frac{\d Q_1}{\d Q},\dots,\frac{\d Q_n}{\d Q}\right)\right|_Q
\end{equation}
for some $P\in \mathcal M_1(\Omega_1)$ dominating $(P_1,\dots,P_n)$ and  $Q\in \mathcal M_1(\Omega_2)$ dominating $(Q_1,\dots,Q_n)$.
\end{defn}

Using the language of heterogeneity order, Lemma \ref{lem} says that in order for compatibility of $(Q_1,\dots,Q_n)\in \mathcal M_1^n$ and $(F_1,\dots,F_n)\in\mathcal F^n$, a necessary condition is $(F_1,\dots,F_n)\lv  (Q_1,\dots,Q_n)$. Before  discussing the sufficiency of this condition, we first establish some properties of heterogeneity order.

The following lemma implies that the choice of the reference measures $P$ and $Q$  in \eqref{eq:question} is irrelevant; in fact, they can be conveniently chosen as the averages of the corresponding measures.
\begin{lem}\label{lem:lem33}
For $(P_1,\dots,P_n)\in \mathcal M_1^n(\Omega_1)$ and $(Q_1,\dots,Q_n)\in \mathcal M_1^n(\Omega_2)$, let $\mathcal M^*_1(\Omega_1)=\{P\in \mathcal M_1(\Omega_1):  (P_1,\dots,P_n)\ll P\}$
and $\mathcal M^*_1(\Omega_2)=\{Q\in \mathcal M_1(\Omega_2): (Q_1,\dots,Q_n)\ll Q\}.$
The following are equivalent:
\begin{enumerate}[(i)]
\item  $(P_1,\dots,P_n)\lv  (Q_1,\dots,Q_n)$; that is, \eqref{eq:question} holds for some $P\in \mathcal M^*_1(\Omega_1)$ and $Q\in \mathcal M^*_1(\Omega_2)$.
\item For  $P=\frac{1}{n}\sum_{i=1}^n P_i$ and $Q=\frac{1}{n}\sum_{i=1}^n Q_i$, \eqref{eq:question} holds.
\item For any $Q\in \mathcal M^*_1(\Omega_2)$, there exists $P\in \mathcal M^*_1(\Omega_1)$ such that
\eqref{eq:question} holds.
\end{enumerate}
\end{lem}

\begin{proof} We proceed in the order (iii)$\Rightarrow$(ii)$\Rightarrow$(i)$\Rightarrow$(iii).

(iii)$\Rightarrow$(ii): For $Q=\frac{1}{n}\sum_{i=1}^n Q_i$, there exists $P^*\in \mathcal M^*_1(\Omega_1)$ such that
$$
\left.\left(\frac{\d P_1}{\d P^*},\dots, \frac{\d P_n}{\d P^*} \right)\right|_{P^*} \lcx \left. \left(\frac{\d Q_1}{\d Q},\dots,\frac{\d Q_n}{\d Q}\right)\right|_Q.
$$

Take the convex function $f:\R^n \to \R$, $f(x_1,\dots,x_n)=(x_1+\cdots+x_n)^2$. It follows from the definition of convex order that
\begin{align*}
\E^{P^*}\left[\left(\frac{\d P_1}{\d P^*}+\dots+\frac{\d P_n}{\d P^*} \right)^2 \right]&\le \E^Q\left[\left(\frac{\d Q_1}{\d Q}+\dots+\frac{\d Q_n}{\d Q}\right)^2\right]\\
&=\E^Q[n^2]=n^2.
\end{align*}

On the other hand,
$$\E^{P^*}\left[ \frac{\d P_1}{\d P^*}+\dots+\frac{\d P_n}{\d P^*} \right]=\E^{P_1}[1]+\dots+\E^{P_n}[1] =n.$$
Hence, $ \frac{\d P_1}{\d P^*}+\dots+\frac{\d P_n}{\d P^*} $ has zero variance under $P^*$,  which implies that it is $P^*$-almost surely equal to $n$.
In other words, $P^*=\frac{1}{n}\sum_{i=1}^n P_i$ on all sets with positive $P^*$-measure. Noting that moreover $P^*$ dominates $(P_1,\dots,P_n)$, we must have $P^*=\frac{1}{n}\sum_{i=1}^n P_i$.
Therefore,  \eqref{eq:question} holds for  $P=\frac{1}{n}\sum_{i=1}^n P_i$ and $Q=\frac{1}{n}\sum_{i=1}^n Q_i$.

(ii)$\Rightarrow$(i): trivial.

(i)$\Rightarrow$(iii):
Assume \eqref{eq:question} holds for some $Q\in \mathcal M_1^*(\Omega_2)$  and $P\in \mathcal M^*_1(\Omega_1)$. Let $\mathbf Y=(\frac{\d Q_1}{\d Q},\dots,\frac{\d Q_n}{\d Q})$, $\mathbf Z=(\frac{\d P_1}{\d P},\dots,\frac{\d P_n}{\d P})$. Let $Q'$ be another probability measure in $\mathcal M_1^*(\Omega_2)$. First, note that without loss of generality, we can assume that $Q'$ is dominated by $Q$. Indeed, any general probability measure $Q'$ can be decomposed as $Q'=cQ_{a}'+(1-c)Q_s'$, where $c\in[0,1]$, $Q'_a$ and $Q'_s$ are probability measures being absolutely continuous and singular with respect to $Q$, respectively. This implies that the distribution of $(\frac{\d Q_1}{\d Q'},\dots,\frac{\d Q_n}{\d Q'})$ is a mixture of the distribution of $c^{-1}(\frac{\d Q_1}{\d Q'_{a}},\dots,\frac{\d Q_n}{\d Q'_{a}})$ (with probability $c$) and $(0,\dots,0)$ (with probability $1-c$). It is easy to check that such a distribution has a larger convex order than $(\frac{\d Q_1}{\d Q'_a},\dots,\frac{\d Q_n}{\d Q'_a})$. Thus, if we show \eqref{eq:question} for $Q'_a$, the result also holds for $Q'$. In the sequel we assume $Q'$ is dominated by $Q$, hence the random variable $X=\frac{\d Q'}{\d Q}$ is well-defined. Let a set $A=\{\mathbf Y \neq 0\}$. Note that since $Q'$ dominates $(Q_1,\dots,Q_n)$, $X>0$ $Q$-almost surely on $A$. $(\frac{\d Q_1}{\d Q'},\dots,\frac{\d Q_n}{\d Q'})$ can be then taken as $X^{-1}\mathbf Y$, where we define $X^{-1}\mathbf Y=0$ when both $X$ and $\mathbf Y$ are 0.

By Lemma \ref{lem:lem2}, there exists a probability space $(\Omega', \mathcal A', \eta)$ and random vectors $\mathbf Y', \mathbf Z'$, such that $\mathbf Y'|_\eta\laweq \mathbf Y|_Q$, $\mathbf Z'|_\eta\laweq \mathbf Z|_P$, and $\E^\eta[\mathbf Y'|\mathbf Z']=\mathbf Z'$. Furthermore, we can obviously choose $(\Omega', \mathcal A', \eta)$ to contain a random variable $X'$ such that $(X',\mathbf Y')|_\eta\laweq (X, \mathbf Y)|_Q$. On $(\Omega', \mathcal A')$, define a new probability measure $\eta'$ by $\frac{\d \eta'}{\d \eta}=X'$, then $(X',\mathbf Y')|_{\eta'}\laweq (X, \mathbf Y)|_{Q'}$. For any bounded measurable function $f$,
$$
\E^\eta[f(\mathbf Z')\mathbf Z']=\E^{\eta}[f(\mathbf Z')\mathbf Y']=\E^{\eta}\left[f(\mathbf Z')\left(\frac{\mathbf Y'}{X'}\right)X'\right]=\E^{\eta'}\left[f(\mathbf Z')\left(\frac{\mathbf Y'}{X'}\right)\right],
$$
where, again, $X'=0$ implies $\mathbf Y'=0$, and in this case $\frac{\mathbf Y'}{X'}$ is set to be 0.
Hence
\begin{align*}
~\E^\eta[f(\mathbf Z')\mathbf Z']&=\E^{\eta'}\left[f(\mathbf Z')\E^{\eta'}\left[\left.\frac{\mathbf Y'}{X'}\right|\mathbf Z'\right]\right]\\
&=\E^\eta\left[f(\mathbf Z')\E^{\eta'}\left[\left.\frac{\mathbf Y'}{X'}\right|\mathbf Z'\right]X'\right]\\
&=\E^\eta\left[f(\mathbf Z')\E^{\eta'}\left[\left.\frac{\mathbf Y'}{X'}\right|\mathbf Z'\right]\E^\eta[X'|\mathbf Z']\right].
\end{align*}

Therefore, we must have
$$
\E^{\eta'}\left[\left.\frac{\mathbf Y'}{X'}\right|\mathbf Z'\right]=\frac{\mathbf Z'}{\E^\eta[X'|\mathbf Z']}
$$
$\eta$-almost surely.
Define measure $P'$ by $\frac{\d P'}{\d P}(z)=\E^\eta[X'|\mathbf Z'=\mathbf Z(z)]=:V(z)$. Note that since
\begin{align*}
\int \frac{\d P'}{\d P}(z)\d P(z)&=\int \E^\eta[X'|\mathbf Z'=\mathbf Z(z)] \d P(z)\\
&=\E^\eta\left[\E^\eta[X'|\mathbf Z']\right]=\E^\eta[X']=\E^Q[X]=1,
\end{align*}
$P'$ is a probability measure. Then we have $(\frac{\d P_1}{\d P'},\dots,\frac{\d P_n}{\d P'})=\frac{\mathbf Z}{V}$. Define probability measure $\eta''$ by $\frac{\d \eta''}{\d \eta}=\E^\eta[X'|\mathbf Z']$. Since the relation between $\mathbf Z'$, $\eta$ and $\eta''$ is in parallel with that between $\mathbf Z$, $P$ and $P'$, we have
$$
\left.\frac{\mathbf Z}{V}\right|_{P'}\laweq \left.\frac{\mathbf Z'}{\E^\eta[X'|\mathbf Z']}\right|_{\eta''}.
$$

However, for any test function $g$,
\begin{align*}
\E^{\eta''}[g(\mathbf Z')]&=\int g(\mathbf Z')\frac{\d \eta''}{\d \eta}\d \eta\\
&=\int g(\mathbf Z')\E^\eta[X'|\mathbf Z']\d \eta=\E^{\eta}[g(\mathbf Z')X']=\E^{\eta'}(g(\mathbf Z')),
\end{align*}
hence $\mathbf Z'|_{\eta'}\laweq \mathbf Z'|_{\eta''}$. Thus, $\frac{\mathbf Z'}{\E^\eta[X'|\mathbf Z']}$, as a function of $\mathbf Z'$, also has the same distribution under $\eta'$ and $\eta''$. Consequently, we have
$$
\left.\left(\frac{\d P_1}{\d P'},\dots,\frac{\d P_n}{\d P'}\right)\right|_{P'}=\left.\frac{\mathbf Z}{V}\right|_{P'}\laweq \left.\frac{\mathbf Z'}{\E^\eta[X'|\mathbf Z']}\right|_{\eta'}.
$$

Also, recalling that $(X',\mathbf Y')|_{\eta'}\laweq (X, \mathbf Y)|_{Q'}$,
$$
\left.\left(\frac{\d Q_1}{\d Q'},\dots,\frac{\d Q_n}{\d Q'}\right)\right|_{Q'}=\left.\frac{\mathbf Y}{X}\right|_{Q'}\laweq \left.\frac{\mathbf Y'}{X'}\right|_{\eta'}.
$$

The proof is finished by noting that
\begin{align*}
&~\E^{\eta'}\left[\left.\frac{\mathbf Y'}{X'}\right|\frac{\mathbf Z'}{\E^\eta[X'|\mathbf Z']}\right]\\
=&~\E^{\eta'}\left[\left.\E^{\eta'}\left[\left.\frac{\mathbf Y'}{X'}\right|\mathbf Z'\right]\right|\frac{\mathbf Z'}{\E^\eta[X'|\mathbf Z']}\right]\\
=&~\frac{\mathbf Z'}{\E^\eta[X'|\mathbf Z']},
\end{align*}
and applying Lemma \ref{lem:lem2} with random vectors $\frac{\mathbf Z'}{\E^\eta[X'|\mathbf Z']}$, $\frac{\mathbf Y'}{X'}$ and measure $\eta'$.
\end{proof}

\begin{rem}
While the definition of  heterogeneity order is given simply by using the convex order between the Radon-Nikodym derivatives of the measures, Lemma \ref{lem:lem33} shows that the choice of the reference measures, hence the exact form of the Radon-Nikodym derivatives, does not affect the order. This explains our motivation to introduce the notion of heterogeneity order as a partial order between two groups of measures rather than between two groups of random variables.
\end{rem}

Some simple and intuitive properties of heterogeneity order are summarized in the following proposition. These properties justify the term ``heterogeneity" in the order $\lh$.

\begin{prop}\label{prop:properties}
For $(P_1,\dots,P_n)\in \mathcal M_1^n(\Omega_1)$ and $(Q_1,\dots,Q_n)\in \mathcal M_1^n(\Omega_2)$, the following holds.
\begin{enumerate}[(i)]
\item If $P_1,\dots,P_n$ are identical, then $(P_1,\dots,P_n)\lv  (Q_1,\dots,Q_n)$.
\item If $Q_1,\dots,Q_n$ are identical, and $(P_1,\dots,P_n)\lv  (Q_1,\dots,Q_n)$, then $P_1$, $\dots$, $P_n$ are also identical.
\item If $Q_1,\dots,Q_n$ are equivalent, and $(P_1,\dots,P_n)\lv  (Q_1,\dots,Q_n)$, then $ P_1,\dots,P_n $ are also equivalent.
\item If $Q_1,\dots,Q_n$ are mutually singular, then $(P_1,\dots,P_n)\lv  (Q_1,\dots,Q_n)$.
 \item If $P_1,\dots,P_n$ are  mutually singular, and $(P_1,\dots,P_n)\lv  (Q_1,\dots,Q_n)$, then $ Q_1,\dots,Q_n $ are also mutually singular.
\end{enumerate}
\end{prop}

\begin{proof}
(i) It is straightforward to verify that
\begin{equation*}
\left.\left(\frac{\d P_1}{\d P_1},\dots, \frac{\d P_n}{\d P_1} \right)\right|_{P_1}\laweq \left.\left(1,\dots,1\right)\right|_{P_1} \lcx \left. \left(\frac{\d Q_1}{\d Q},\dots,\frac{\d Q_n}{\d Q}\right)\right|_Q
\end{equation*}
for any $Q\in \mathcal M_1(\Omega_2)$ that dominates $(Q_1,\dots,Q_n)$. As a result, we have $(P_1,\dots,P_n)\lv  (Q_1,\dots,Q_n)$.

(ii) By  $(P_1,\dots,P_n)\lv  (Q_1,\dots,Q_n)$ and Lemma \ref{lem:lem33}, we have
\begin{equation}\label{eq:h-o-2}
\left.\left(\frac{\d P_1}{\d P},\dots, \frac{\d P_n}{\d P} \right)\right|_P \lcx \left. \left(1,\dots,1\right)\right|_{Q_1}
\end{equation}
holds  for some $P\in \mathcal M_1(\Omega_1)$ dominating $(P_1,\dots,P_n)$.
By Lemma \ref{lem:lem2}, \eqref{eq:h-o-2}  further implies $\d P_i/\d P =1$ $P$-almost surely for $i=1,\dots,n$;
thus $P_1,\dots,P_n$ are identical.

(iii) Let $P=\frac{1}{n}\sum_{i=1}^n P_i$ and $Q= \frac{1}{n}\sum_{i=1}^n Q_i$.
$(P_1,\dots,P_n)\lv  (Q_1,\dots,Q_n)$ implies that,  for each $i=1,\dots,n$,
$$\frac{\d P_i}{\d P}\Big|_P\lcx \frac{\d Q_i}{\d Q}\Big|_Q.$$
Note that $Q(\d Q_i/\d Q=0)=0$ as $Q_1,\dots,Q_n$ are equivalent. By Lemma \ref{lem:lem2}, we know  $P(\d P_i/\d P=0)=0$, which implies $P\ll P_i$.
Thus, $P_1,\dots,P_n$ are equivalent.

 (iv) As $Q_1,\dots,Q_n$ are mutually singular,  there exists a partition $\{\Omega_1,\dots,\Omega_n\}$ $\subseteq$ $\mathcal A$ of $\Omega$ such that $Q_i(\Omega_i)=1$, $i=1,\dots,n$.
Let $P=\frac{1}{n}\sum_{i=1}^n P_i$ and $Q= \frac{1}{n}\sum_{i=1}^n Q_i$.
 Note that
 $$ \left(\frac{\d Q_1}{\d Q},\dots,\frac{\d Q_n}{\d Q}\right)= n\times \left(\id_{\Omega_1},\dots,\id_{\Omega_n}\right)$$
 takes values in the vertices of the simplex
 $$
 S=\left\{(s_1,\dots,s_n)\in \R_+^n: \sum_{i=1}^n s_i=n\right\},
 $$
 and $(\frac{\d P_1}{\d P},\dots, \frac{\d P_n}{\d P} )$  takes values in $S$.
 Furthermore,
 $$
 \E^P\left[\left(\frac{\d P_1}{\d P},\dots, \frac{\d P_n}{\d P} \right)\right]=(1,\dots,1)=\E^Q\left[\left(\frac{\d Q_1}{\d Q},\dots,\frac{\d Q_n}{\d Q} \right)\right].
 $$
 By the Choquet-Meyer Theorem (\cite{CM63}; see Section 10 of \cite{P01}), stating that among random vectors distributed in a simplex, the maximal elements with respect to convex order are supported over the vertices of the simplex,   we have
 $$\left.\left(\frac{\d P_1}{\d P},\dots, \frac{\d P_n}{\d P} \right)\right|_P \lcx \left. \left(\frac{\d Q_1}{\d Q},\dots,\frac{\d Q_n}{\d Q}\right)\right|_Q.
$$

 (v) Using the notation in (iv),
 $(\frac{\d P_1}{\d P},\dots, \frac{\d P_n}{\d P} )$
 takes values in the vertices of the simplex $S$,
 and $(\frac{\d Q_1}{\d Q},\dots, \frac{\d Q_n}{\d Q} )$  takes values in $S$.
 Therefore, by the Choquet-Meyer Theorem again, in order for $(P_1,\dots,P_n)\lh (Q_1,\dots,Q_n)$ to hold,
$(\frac{\d Q_1}{\d Q},\dots, \frac{\d Q_n}{\d Q} )$ has to be distributed over the vertices of the simplex $S$, and therefore, $Q_1,\dots,Q_n$ are mutually singular.
\end{proof}

\subsection{Almost compatibility}\label{sec:sec33}

In Section \ref{sec:sec32}, we see that  a necessary condition for compatibility of $(Q_1,\dots,Q_n)$ $\in$ $\mathcal M_1^n$ and $(F_1,\dots,F_n)\in\mathcal F^n$ is $(F_1,\dots,F_n)\lv  (Q_1,\dots,Q_n)$.  A natural question is whether (and with what additional assumptions) the above condition is also sufficient. This boils down (via Theorem \ref{th:newstuff}) to the question of, given
$$
\left.\left(\frac{\d F_1}{\d F},\dots, \frac{\d F_n}{\d F} \right)\right|_{F} \lcx \left. \left(\frac{\d Q_1}{\d Q},\dots,\frac{\d Q_n}{\d Q}\right)\right|_Q,
$$
where $F=\frac 1n\sum_{i=1}^n F_i$ and $Q=\frac 1n\sum_{i=1}^n Q_i$,
constructing a random variable  $X$ with distribution $F$ under $Q$  such that
\begin{equation}\label{eq:existx}
 \left(\frac{\d F_1}{\d F},\dots, \frac{\d F_n}{\d F} \right)(X) =   \E^Q\left[\left.\left(\frac{\d Q_1}{\d {Q}},\dots,\frac{\d Q_n}{\d {Q}}\right)\right|X\right].
\end{equation}

Such problem is similar to Lemma \ref{lem:lem2}, and more generally, the martingale construction in  \cite{S65} or \cite{HPRY11}, albeit we need to construct $X$ in the pre-specified space
$(\Omega,\mathcal A,Q)$.
Therefore, the existence of $X$ satisfying \eqref{eq:existx} naturally depends on the probability space $(\Omega,\mathcal A,Q)$.
As a simple example, if $F$ is a continuous distribution and one of $Q_1,\dots,Q_n$ is not atomless, then there does not exist a random variable $X$ with distribution $F$ under each of $Q_1,\dots,Q_n$, although $(F,\dots,F)\lh (Q_1,\dots,Q_n)$ by Proposition \ref{prop:properties} (i).

It seems then natural to assume that each of $Q_1,\dots,Q_n$ is atomless.
  Below we give a counter example showing that this condition is still insufficient.

 \begin{exam}\label{ex:ct1}
Let $\Omega=[0,1]$, $\mathcal A=\mathcal B([0,1])$, $Q_2=\lambda$ be the Lebesgue measure,  $\frac{\d Q_1}{\d Q_2}(t)=2t,$ $t\in [0,1]$,   $F_2=\lambda$ on $[0,1]$ and $\frac{\d F_1}{\d F_2}(x)=|4x-2|,$ $x\in [0,1]$.
For this setting we have $(F_1,F_2)\lv (Q_1,Q_2)$ but $(F_1,F_2)$ and $(Q_1,Q_2)$ are not compatible. The details of these statements are given in Appendix \ref{app:a1}.
\end{exam}

Example \ref{ex:ct1} suggests that the atomless condition, combined with the heterogeneity order $(F_1,\dots,F_n)\lv  (Q_1,\dots,Q_n)$,
is not sufficient for compatibility of  $ (Q_1,\dots,Q_n)$ and $(F_1,\dots,F_n) $.
Nevertheless, in this section we show that, assuming $Q_1,\dots,Q_n$ are atomless, $(F_1,\dots,F_n)\lv  (Q_1,\dots,Q_n)$
 is sufficient for \emph{almost compatibility}, a weaker notion than compatibility, which we introduce below.
 Denote by $D_{\text{KL}}(\cdot\Vert\cdot)$ the Kullback-Leibler divergence between probability measures. Recall that $D_{\text{KL}}(P\Vert Q)$ is defined as $\int \log(\d P/\d Q)\d P$ for $P\ll Q$.

%
\begin{defn}\label{def:almost}
$(Q_1,\dots,Q_n)\in \mathcal M_1^n$ and $(F_1, \dots, F_n)\in \mathcal F^n$ are \emph{almost compatible}, if for any $\epsilon>0$, there exists a random variable $X_\epsilon$ in $(\Omega,\mathcal A)$ such that
for each $i=1,\dots,n$, the distribution of $X_\epsilon$ under $Q_i$, denoted by $ F_{i,\epsilon}$, is absolutely continuous with respect to $F_i$, and satisfies $D_{\text{KL}}(F_{i,\epsilon}\Vert F_{i})<\epsilon$.
\end{defn}

The following theorem characterizes almost compatibility via heterogeneity order in Definition \ref{def2},
assuming each of $Q_1,\dots,Q_n$ is atomless.

 \begin{theorem}\label{th:almost:compatible}
 Suppose that $(Q_1,\dots,Q_n)\in \mathcal M_1^n$,  $(F_1,\dots,F_n)\in \mathcal F^n$ and each of $Q_1,\dots,Q_n$ is atomless. $(Q_1,\dots,Q_n)$ and $(F_1,\dots,F_n)$ are almost compatible if and only if $(F_1,\dots,F_n)\lv  (Q_1,\dots,Q_n)$.
\end{theorem}

The proof of Theorem \ref{th:almost:compatible} is a bit lengthy, and is postponed to  Appendix \ref{app:a2} of the paper.

\begin{rem}\label{rem:32}
The Kullback-Leibler divergence in Definition \ref{def:almost} is not the only possible choice to provide an equivalent condition in Theorem \ref{th:almost:compatible}. Indeed, the condition for necessity can be weakened to the convergence in probability of $\d F_{i,\epsilon}/\d F_i$ to 1 as $\epsilon\to 0$, by using Fatou's lemma and the fact that a sequence converging in probability has a subsequence converging almost surely; the proof for sufficiency implies results as strong as the uniform convergence of $\d F_{i,\epsilon}/\d F_i$ to 1. Consequently, the Kullback-Leibler divergence used in the definition of the almost compatibility can be replaced by a series of other conditions, including:
\begin{enumerate}[(i)]
\item $\d F_{i,\epsilon}/\d F_i\stackrel{\mathrm p}{\to} 1$;
\item $\d F_{i,\epsilon}/\d F_i\stackrel{\mathrm {a.s.}}{\to} 1$;
\item $F_{i,\epsilon}$ converges to $F_i$ in total variation, and $F_{i,\epsilon}\ll F_i$;
\item The R\'{e}nyi divergence of order $\infty$ between $F_{i,\epsilon}$ and $F_i$ converges to $0$ as $\epsilon\to 0$,
\end{enumerate}
among others, without altering the result of Theorem \ref{th:almost:compatible}.
\end{rem}

Almost compatibility has a practical implication for optimization problems.
Suppose that $Q_1,\dots,Q_n$ are atomless. For optimization problems of the form
\begin{equation*} 
  \sup\{\phi (P\circ Y^{-1}):Y\in \mathcal L^0(\Omega;\mathbb R) \mbox{~has distribution $F_i$ under $Q_i$,~$i=1,\dots,n$}\},
  \end{equation*}
  where $\phi : \mathcal F\to [-\infty,\infty]$ is a functional,
  it suffices to consider
$$
  \sup\left\{\phi(F): F\in \mathcal F,~(F_1,\dots,F_n,F)\lv (Q_1,\dots,Q_n,P)\right\},
$$
  as long as $\phi$ is continuous with respect to any of the convergence types listed in Remark \ref{rem:32}.

 \subsection{Equivalence of heterogeneous order and compatibility}\label{sec:compatibility}

 In view of the discussions in Section \ref{sec:sec33}, $(F_1,\dots,F_n)\lh (Q_1,\dots,Q_n)$ is not sufficient for compatibility of $(Q_1,\dots,Q_n)$ and $(F_1,\dots,F_n)$, but sufficient for almost compatibility if   each of $Q_1,\dots,Q_n$  is atomless.
 In this section, we seek for a
  slightly stronger condition on the $n$-tuple $(Q_1,\dots,Q_n)$, under which compatibility and almost compatibility coincide.

 \begin{defn}\label{def:condatomless}
$(Q_1,\dots,Q_n)\in \mathcal M_1^n$ is \emph{conditionally atomless} if there exist  $Q\in\mathcal M_1$ dominating $(Q_1, \dots, Q_n)$ and  $X\in \mathcal L^0(\Omega;\mathbb R)$ such that under $Q$, $X$ is continuously distributed and independent of $(\frac{\d Q_1}{\d Q},\dots,\frac{\d Q_n}{\d Q})$.
 \end{defn}

Clearly, if $(Q_1,\dots,Q_n) $ is {conditionally atomless}, then each of $Q_1,\dots,Q_n$ is atomless, since a continuous random variable under $Q$ is also continuous under each $Q_1,\dots,Q_n$.

 \begin{rem}\label{rem:new33}
If $Q_1,\dots,Q_n$ are mutually singular and each of them is atomless, then $(Q_1,\dots,Q_n)$ is conditionally atomless. This can be seen directly by constructing a uniform random variable $U_i$ on $[0,1]$ under $Q_i$ for $i=1,\dots,n$, and writing $Q=\frac 1n\sum_{i=1}^n Q_i$.
As $Q_1,\dots,Q_n$ are mutually singular,  there exists a partition $\{\Omega_1,\dots,\Omega_n\}\subset \mathcal A$ of $\Omega$ such that $Q_i(\Omega_i)=1$, $i=1,\dots,n$.
Then the random variable $U=\sum_{i=1}^n U_i \id_{\Omega_i}$ is uniformly distributed and independent of $(\frac{\d Q_1}{\d Q},\dots,\frac{\d Q_n}{\d Q})$ under $Q$.
\end{rem}

Before  approaching the main results of this section, we  recall some basic facts about conditional distributions.
For random vectors $\mathbf T$ and $\mathbf S$ defined on a probability space $(\Omega, \mathcal A, P)$ and taking values in $\mathbb R^m$ and $\mathbb R^n$, respectively, the conditional distribution of $\mathbf T$ given $\mathbf S$ (under $P$), denoted by $\mathbf T|\mathbf S$,   is a mapping from $\mathcal B(\mathbb R^m)\times \Omega$ to $\mathbb R$, such that for each $\omega\in\Omega$, $\mathbf T|\mathbf S(\cdot,\omega)$ is a probability measure on $(\mathbb R^m, \mathcal B(\mathbb R^m))$, and for each $A\in\mathcal B(\mathbb R^m)$, $\mathbf T|\mathbf S(A,\cdot)=P(\mathbf T\in A|\sigma(\mathbf S))$ $P$-almost surely. We write $\mathbf T|\mathbf S(\omega)$ for the probability measure $\mathbf T|\mathbf S(\cdot, \omega)$, and $\mathbf T|\mathbf S(\omega)_P$ when it is necessary to specify the probability measure $P$. Moreover, there exists a version of $\mathbf T|\mathbf S$ for which the conditional distribution only depends on the value of $\mathbf S$, i.e., $\mathbf T|\mathbf S(\omega_1)=\mathbf T|\mathbf S(\omega_2)$ whenever $\mathbf S(\omega_1)=\mathbf S(\omega_2)$. We will always use this version. For an event $E\in\mathcal A$, the conditional probability of $E$ given $\mathbf S=s$, denoted by $P(E|\mathbf S=s)$, should be understood as $P[E|\sigma(\mathbf S)](\omega)$ for $\omega$ satisfying $\mathbf S(\omega)=s$.

With the help of conditional distributions, we first note that the independence in Definition \ref{def:condatomless} is not essential and can be replaced by continuity of the conditional distribution. Moreover, similarly to  heterogeneity order, the reference probability measure $Q$ can always be taken as $Q=\frac{1}{n}\sum_{i=1}^nQ_i$.

\begin{prop}\label{prop:conatomless}
For $(Q_1,\dots,Q_n)\in \mathcal M_1^n$, the following are equivalent:
\begin{enumerate}[(i)]
\item $(Q_1,\dots, Q_n)$ is conditionally atomless.\label{con.atomless}
\item For $Q=\frac{1}{n}\sum_{i=1}^nQ_i$, there exists a continuous random variable in $(\Omega,\mathcal A)$  independent of $(\frac{\d Q_1}{\d Q},\dots,\frac{\d Q_n}{\d Q})$ under $Q$.\label{con.averageQ}
\item There exists   $X\in \mathcal L^0(\Omega;\mathbb R)$ such that for some $Q\in\mathcal M_1$ which dominates $(Q_1,\dots,Q_n)$ (equivalently, for $Q=\frac{1}{n}\sum_{i=1}^nQ_i$), a version of the conditional distribution $ X|\mathbf Y$ is everywhere continuous under $Q$ where $\mathbf Y=(\frac{\d Q_1}{\d Q},\dots,\frac{\d Q_n}{\d Q})$.\label{con.cont}
\end{enumerate}
\end{prop}

\begin{proof}
Note that (\ref{con.cont}) has two versions: one states the existence of $Q$ and the other specifies $Q$.
It is trivial to see that (\ref{con.averageQ}) implies (\ref{con.atomless}) and both versions of (\ref{con.cont}).  It remains to show (\ref{con.cont})$\Rightarrow$(\ref{con.atomless})$\Rightarrow$(\ref{con.averageQ}).

We first show (\ref{con.atomless})$\Rightarrow$(\ref{con.averageQ}). Assume $(Q_1,\dots, Q_n)$ is conditionally atomless. As a result, there exist $Q'\in\mathcal M_1$ and a random variable $X$, such that $X$ and $\mathbf Y:=(\frac{\d Q_1}{\d Q'},\dots,\frac{\d Q_n}{\d Q'})$ are independent under $Q'$.
For $i=1,\dots,n$, $A\in\mathcal B(\mathbb R)$ and $B\in\mathcal B(\mathbb R^n)$,
\begin{align*}
Q_i(X\in A, \mathbf Y\in B)&=\E^{Q'}\left[\frac{\d Q_i}{\d Q'}\id_{\{X\in A\}}\id_{\{\mathbf Y\in B\}}\right]\\
&=\E^{Q'}[\id_{\{X\in A\}}]\E^{Q'}\left[\frac{\d Q_i}{\d Q'}\id_{\{\mathbf Y\in B\}}\right]=Q'(X\in A)Q_i(\mathbf Y\in B).
\end{align*}

The independence between $X$ and $\mathbf Y$ also implies that
$$
Q_i(X\in A)=\E^{Q'}\left[\frac{\d Q_i}{\d Q'}\id_{\{X\in A\}}\right]=\E^{Q'}\left[\frac{\d Q_i}{\d Q'}\right]\E^{Q'}[\id_{\{X\in A\}}]=Q'(X\in A).
$$

Thus, $X$ has the same distribution under $Q_i$, $i=1,\dots, n$. Let $Q=\frac{1}{n}\sum_{i=1}^nQ_i$, and note that $X$ also has the same distribution under $Q$. Moreover,
$$
Q_i(X\in A, \mathbf Y\in B)=Q'(X\in A)Q_i(\mathbf Y\in B)=Q_i(X\in A)Q_i(\mathbf Y\in B),
$$
which means that $X$ and $\mathbf Y$ are independent under $Q_i$ for $i=1,\dots, n$. For any $A\in\mathcal B(\mathbb R)$ and $B\in\mathcal B(\mathbb R^n)$,
\begin{align*}
Q(X\in A, \mathbf Y\in B)&=\frac{1}{n}\sum_{i=1}^n Q_i(X\in A, \mathbf Y\in B)\\
&=\frac{1}{n}\sum_{i=1}^n Q_i(X\in A)Q_i(\mathbf Y\in B)\\
&=Q(X\in A)\frac{1}{n}\sum_{i=1}^n Q_i(\mathbf Y\in B)=Q(X\in A)Q(\mathbf Y\in B),
\end{align*}
and hence $X$ and $\mathbf Y$ are independent under $Q$. As a result, $X$ is also independent of
$$
\frac{\mathbf Y}{\Vert\mathbf Y\Vert_1}=\frac{1}{n}\left(\frac{\d Q_1}{\d Q},\cdots, \frac{\d Q_n}{\d Q}\right)
$$
under $Q$, where $\Vert\cdot\Vert_1$ is the Manhattan norm on $\R^n$. Therefore, we conclude that $X$ and $(\frac{\d Q_1}{\d Q},\cdots, \frac{\d Q_n}{\d Q})$ are independent under $Q$.

Next we prove  (\ref{con.cont})$\Rightarrow$(\ref{con.atomless}). Take $X$ as in (\ref{con.cont}) and let $F_\omega$ be the distribution function of $ X|\mathbf Y(\omega)$, and define $X': \Omega\to\mathbb R$ by $X'(\omega)=F_\omega(X(\omega))$. It is fundamental, though a bit lengthy, to check that $X'$ is a random variable; moreover, $X'|\mathbf Y$ almost surely follows a uniform distribution on $[0,1]$. As a result, $X'$ is a continuous random variable independent of $(\frac{\d Q_1}{\d Q},\dots,\frac{\d Q_n}{\d Q})$ under $Q$. Consequently, both versions of (\ref{con.cont}) imply (\ref{con.atomless}).
\end{proof}

\begin{rem}\label{re:samedist}
As a byproduct of the above proof, we note that if a random variable $X$ is independent of $(\frac{\d Q_1}{\d Q},\dots,\frac{\d Q_n}{\d Q})$ under a probability measure $Q$, then $X$ is also independent of $(\frac{\d Q_1}{\d Q},\dots,\frac{\d Q_n}{\d Q})$ under each of $Q_1,\dots, Q_n$. Moreover, $X$ has the same distribution under $Q_1,\dots,Q_n$ and $Q$.
\end{rem}

\begin{rem}
Right before the publication of this paper,  a new preprint \cite{D19} introduces the concept of a conditionally atomless $\sigma$-field 
 which turns out to be closely related to our notion of  conditionally atomless   measures in Definition \ref{def:condatomless}. For the connection and the differences between the two formulations, see the discussions in \cite{D19}\footnote{We thank Freddy Delbaen for pointing out the preprint and for  very useful discussions.}.
\end{rem}

Now we   turn back to our main target, compatibility of $(F_1,\dots,F_n)$ and $(Q_1,\dots,Q_n)$.
As discussed in Section \ref{sec:sec33}, to show compatibility one  needs to construct a random variable $X$ in $(\Omega,\mathcal A)$ such that
 $$
 \left(\frac{\d F_1}{\d F},\dots, \frac{\d F_n}{\d F} \right)(X) =   \E^Q\left[\left.\left(\frac{\d Q_1}{\d {Q}},\dots,\frac{\d Q_n}{\d {Q}}\right)\right|X\right].
 $$
 It turns out that the assumption that $(Q_1,\dots,Q_n)$ is conditionally atomless allows for such a construction.

\begin{theorem}\label{th:meas.variable}
Suppose that $(Q_1,\dots,Q_n)\in \mathcal M_1^n$ is conditionally atomless and $(F_1,\dots,F_n)\in\mathcal F^n$.
$(Q_1,\dots,Q_n)$ and $(F_1,\dots,F_n)$ are compatible if and only if  $(F_1,\dots,F_n)\lv (Q_1,\dots,Q_n)$.
\end{theorem}

The key step to prove Theorem \ref{th:meas.variable} is the following lemma, which might be of independent interest.

\begin{lem}\label{lem:transplant}
Let $\mathbf X=(X_1,\dots,X_m)$ and $\mathbf Y=(Y_1,\dots,Y_n)$ be random vectors defined on probability spaces $(\Omega_1, \mathcal A_1, P_1)$ and $(\Omega_2, \mathcal A_2, P_2)$, respectively, and $f$ be a measurable function from $(\mathbb R^m, \mathcal B(\mathbb R^m))$ to $(\mathbb R^n, \mathcal B(\mathbb R^n))$. If the convex order relation $f(\mathbf X)|_{P_1}\lcx \mathbf Y|_{P_2}$ holds, and   there exists a continuous random variable $U$ defined on $(\Omega_2, \mathcal A_2, P_2)$ independent of $\mathbf Y$, then there exists a random vector $\mathbf W=(W_1,\dots, W_m)$ defined on $(\Omega_2, \mathcal A_2, P_2)$, such that $\mathbf W|_{P_2} \laweq \mathbf X|_{P_1}$, and
$$
f(\mathbf W)=\E^{P_2}[\mathbf Y|\mathbf W].
$$
\end{lem}

\begin{proof}
Since $f(\mathbf X)|_{P_1}\lcx \mathbf Y|_{P_2}$, by Lemma \ref{lem:lem2}, there exists a probability space $(\Omega', \mathcal A', P')$ and random vectors $\mathbf Z$, $\mathbf Y'$ defined on it and taking values in $\mathbb R^n$, such that $\mathbf Z|_{P'}\laweq f(\mathbf X)|_{P_1}$, $\mathbf Y'|_{P'}\laweq \mathbf Y|_{P_2}$, and $\mathbf Z=\E^{P'}[\mathbf Y'|\mathbf Z]$.

Construct random vectors $\mathbf X''=(X''_1,\dots,X''_m)$ and $\mathbf Y''=(Y''_1,\dots,Y''_n)$ on a (possibly different) probability space $(\Omega'', \mathcal A'', P'')$, such that $\mathbf X''|_{P''}\laweq \mathbf X|_{P_1}$ and the conditional distributions satisfy $\mathbf Y''|\mathbf X''(\omega'')_{P''}=\mathbf Y'|\mathbf Z(\omega')_{P'}$ for all $\omega',\omega''$ satisfying $\mathbf Z(\omega')=f(\mathbf X''(\omega''))$ . It is easy to see that $\mathbf Y''|_{P''}\laweq \mathbf Y|_{P_2}$, and
$$
\E^{P''}[\mathbf Y''|\mathbf X''](\omega'')=\E^{P'}[\mathbf Y'|\mathbf Z](\omega')=\mathbf Z(\omega')=f(\mathbf X''(\omega'')), \mbox{~~for $P''$-a.s.~} \omega'' \in \Omega''.
$$

What is left is therefore to construct a random vector $\mathbf W$ on $(\Omega_2,\mathcal A_2, P_2)$ such that $(\mathbf W,\mathbf Y)|_{P_2}\laweq (\mathbf X'', \mathbf Y'')|_{P''}$. The idea is similar to the proof of Theorem \ref{th:almost:compatible}. Indeed, for $\ell=0,1,\dots$ and $h=(h_1,\dots,h_m)\in {\mathbb Z}^{m}$, consider the distribution of $\mathbf Y''$ restricted on the event $\{X_i''\in [h_i2^{-\ell}, (h_i+1)2^{-\ell}), ~1\leq i\leq m\}$. {It has a density function, denoted by $\psi_{\ell,h}(\mathbf y)$, $\mathbf y\in \mathbb R^n$, with respect to the unconditional distribution of $\mathbf Y''$}. Without loss of generality, assume $U$ follows a uniform distribution on $[0,1]$. Then for each $\mathbf y$ and $\ell=0,1,\dots$, we divide $[0,1]$ into disjoint intervals $\{I_{\ell,h}(\mathbf y)\}_{h\in {\mathbb Z}^{m}}$, such that $|I_{\ell,h}(\mathbf y)|=\psi_{\ell,h}(\mathbf y)$. Moreover, we can make $\{I_{\ell',h}(\mathbf y)\}_{h\in {\mathbb Z}^{m}}$ a refinement of $\{I_{\ell,h}(\mathbf y)\}_{h\in {\mathbb Z}^{m}}$ for any $\ell'>\ell$. Then define random vector $\mathbf W_\ell=(W_{\ell,1},\dots,W_{\ell,m})$ by
$$
W_{\ell,i}=h_i2^{-\ell} \text{ for } U\in I_{\ell,h}(\mathbf Y), \quad i=1,\dots,m.
$$
Let $\mathbf W=\lim_{\ell\to\infty}\mathbf W_\ell$. The point-wise limit exists due to the completeness of ${\mathbb R}^{m}$.

For any given $\mathbf y$, any $\ell=0,1,\dots$ and $h\in\mathbb Z^m$,
\begin{align*}
& P_2(\mathbf W_i\in[h_i2^{-\ell}, (h_i+1)2^{-\ell}), ~1\leq i\leq m|\mathbf Y=\mathbf y)\\
&= P_2(\mathbf W_{\ell,i}=h_i2^{-\ell}, ~1\leq i\leq m|\mathbf Y=\mathbf y)\\
&= \psi_{\ell,h}(\mathbf y)= P''(\mathbf X''_i\in[h_i2^{-\ell}, (h_i+1)2^{-\ell}), ~1\leq i\leq m|\mathbf Y''=\mathbf y).
\end{align*}
Since $\{[h_i2^{-\ell}, (h_i+1)2^{-\ell})\}_{h\in \mathbb Z^m, \ell=0,1,\dots}$ forms a basis for $\mathcal B(\mathbb R^m)$, we have $\mathbf W|\mathbf Y(\omega)$ under $P_2$ equals $\mathbf X''|\mathbf Y''(\omega'')$ for any $\omega\in\Omega$ and $\omega''\in\Omega''$ satisfying $\mathbf Y(\omega)=\mathbf Y''(\omega'')$. Moreover, recall that $\mathbf Y|_{P_2}\laweq \mathbf Y''|_{P''}$. As a result, we conclude that $(\mathbf W, \mathbf Y)|_{P_2}\laweq (\mathbf X'', \mathbf Y'')_{P''}$.
\end{proof}

\begin{proof}[Proof of Theorem \ref{th:meas.variable}]

Necessity is guaranteed by Lemma \ref{lem}. We only show sufficiency.
 Suppose that $(F_1,\dots,F_n)\lv (Q_1,\dots,Q_n)$. We shall show that $(Q_1,\dots,Q_n)$ and $(F_1,\dots,F_n)$ are compatible.
By Lemma \ref{lem:lem33},
$$
\left.\left(\frac{\d F_1}{\d F},\dots, \frac{\d F_n}{\d F} \right)\right|_F \lcx \left. \left(\frac{\d Q_1}{\d Q},\dots,\frac{\d Q_n}{\d Q}\right)\right|_Q
$$
for $F=\frac{1}{n}\sum_{i=1}^n F_i$ and $Q=\frac{1}{n}\sum_{i=1}^n Q_i$. Since $(Q_1,\dots,Q_n)$ is conditionally atomless, $Q_1,\dots, Q_n$ are all atomless, so is $Q$. Hence there exists a random variable $X'$ defined on $(\Omega, \mathcal A)$, such that $F=Q\circ {X'}^{-1}$. As a result,
$$
  \left.\left(\frac{\d F_1}{\d F},\dots, \frac{\d F_n}{\d F} \right)(X')\right|_Q\laweq \left.\left(\frac{\d F_1}{\d F},\dots, \frac{\d F_n}{\d F} \right)\right|_F \lcx \left. \left(\frac{\d Q_1}{\d Q},\dots,\frac{\d Q_n}{\d Q}\right)\right|_Q.
$$

Applying Lemma \ref{lem:transplant} with $f(x)=(\frac{\d F_1}{\d F},\dots, \frac{\d F_n}{\d F})(x)$, there exists a random variable $X$ defined on $(\Omega, \mathcal A)$, such that
$$
\left(\frac{\d F_1}{\d F},\dots, \frac{\d F_n}{\d F} \right)(X)=\E^Q \left[\left. \left(\frac{\d Q_1}{\d Q},\dots,\frac{\d Q_n}{\d Q}\right)\right|X\right],
$$
which, by Theorem \ref{th:newstuff}, implies compatibility.
\end{proof}

\begin{rem}
As shown in Theorem \ref{th:meas.variable}, compatibility is closely related to heterogeneity order $\lh$, and hence it defines a partial order. The direction of the order comes from the fact that a measurable mapping needs not to be a bijection. As multiple points are mapped to a same image, the ``heterogeneity'' between measures decreases. However, if we require the mapping to be a bijection, then compatibility  becomes an equivalence relation. Indeed, in this case Theorem \ref{th:meas.variable} would be applicable to both directions, which means that (\ref{eq:question}) holds for both directions, with $P=\frac{1}{n}\sum_{i=1}^n P_i$ and $Q=\frac{1}{n}\sum_{i=1}^n Q_i$. As a result, we must have
$$
\left.\left(\frac{\d P_1}{\d P},\dots, \frac{\d P_n}{\d P} \right)\right|_P \laweq \left. \left(\frac{\d Q_1}{\d Q},\dots,\frac{\d Q_n}{\d Q}\right)\right|_Q.
$$
Moreover, the proof of Theorem \ref{th:meas.variable} actually shows that, assuming both tuples of measures are conditionally atomless, the above condition is not only necessary but also sufficient to guarantee the existence of a bijection linking $(P_1,\dots, P_n)$ to $(Q_1,\dots, Q_n)$.
\end{rem}

\begin{rem}\label{cor:37}
As a simple consequence of Theorem \ref{th:meas.variable}, in the case where $n=2$ and $Q_1\ll Q_2$, if $(Q_1,Q_2)$ and $(F_1,F_2)$ are compatible, then $F_1\ll F_2$ and
\begin{equation*}
\left. \frac{\d F_1}{\d F_2} \right|_{F_2} \lcx \left.  \frac{\d Q_1}{\d Q_2}\right|_{Q_2}.
\end{equation*}
The converse is also true if, in addition, $(Q_1,Q_2)$ is conditionally atomless. Therefore, the heterogeneity order condition becomes one-dimensional, and is easy to check.   Chapter 3 of \cite{SS07} contains several classic methods to check $X|_P\lcx Y|_Q$ for arbitrary random variables $X$ and $Y$ and probability measures $P$ and $Q$.
\end{rem}


%
%
%
%
%

  Below we discuss  a few special cases of  compatible $(Q_1,\dots,Q_n)\in\mathcal M_1^n$ and $(F_1,\dots,F_n)\in\mathcal F^n$   based on the heterogeneity order condition, in particular in the context of Proposition \ref{prop:properties} and Theorem \ref{th:meas.variable}. We shall see how our main results are consistent with natural intuitions. 

 1. Assume that $Q_1,\dots,Q_n$ are identical. The natural intuition  is that the respective distributions $F_1,\dots,F_n$ of a random variable under  $Q_1,\dots,Q_n$ have to be identical as well. Indeed, by Lemma \ref{lem}, compatibility implies $(F_1,\dots,F_n)\lh (Q_1,\dots,Q_n)$. By Proposition \ref{prop:properties} (ii), $ F_1,\dots,F_n$ are identical.

2. Assume that $Q_1,\dots,Q_n$ are mutually singular, and each of them is atomless. The natural intuition here is that the respective distributions $F_1,\dots,F_n$ of any random variable under   $Q_1,\dots,Q_n$ are arbitrary.
Proposition \ref{prop:properties} (iv) suggests that  $(F_1,\dots,F_n)\lh (Q_1,\dots,Q_n)$ holds for any $(F_1,\dots,F_n)\in \mathcal F^n$.
Moreover, $(Q_1,\dots,Q_n)$ is conditionally atomless, as seen in Remark \ref{rem:new33}.
Therefore, by Theorem \ref{th:meas.variable}, a mutually singular tuple of atomless probability measures on $ (\Omega,\mathcal A)$ is compatible with an arbitrary tuple of distributions on $\R$.

3. Assume that $F_1,\dots,F_n$ are mutually singular.
 The natural intuition here is that the probability measures $Q_1,\dots,Q_n$ have to be also mutually singular to allow for compatibility.
  Similarly to the previous case, this is justified by Theorem \ref{th:meas.variable} and Proposition \ref{prop:properties} (v).

4. Assume that $F_1,\dots,F_n$ are identical, and $(Q_1,\dots,Q_n)$ is conditionally atomless.
Proposition \ref{prop:properties} (i) gives $ (F_1,\dots,F_n)\lh (Q_1,\dots,Q_n)$.
It follows from Theorem \ref{th:meas.variable} that $(Q_1,\dots,Q_n)$ and $(F_1,\dots,F_n)$ are compatible.
We conclude that, as long as $(Q_1,\dots,Q_n)$ is conditionally atomless, for any distribution $F\in\mathcal F$, there exists a random variable $X$ which has distribution $F$ under each of $Q_i$, $i=1,\dots,n$. Indeed, as $(Q_1,\dots,Q_n)$ is conditionally atomless, there exists $Q$ dominating $(Q_1,\dots,Q_n)$ and an $F$-distributed random variable $X$ under $Q$ independent of $(\frac{\d Q_1}{\d Q},\dots,\frac{\d Q_n}{\d Q})$. Remark \ref{re:samedist} then implies that $X$ also has distribution $F$ under each $Q_1,\dots,Q_n$.

5. Assume that $Q_1,\dots,Q_n$ are equivalent. Intuitively, the respective distributions $F_1,\dots,F_n$ of any random variable under  $Q_1,\dots,Q_n$ have to be equivalent.
This fact is implied by Proposition \ref{prop:properties} (iii).

\begin{rem}\label{re:link}
A notion similar to  heterogeneity order is
useful
 in comparison of statistical experiments, an area of study originated by Blackwell (\cite{B51, B53}); the interested reader is referred to \cite{L96} and \cite{T91} for summaries.
\end{rem}

\section{Distributional compatibility for stochastic processes}\label{sec:processes}

\subsection{General results}

In this section we extend our results to stochastic processes with sample paths which are continuous from right with left limits (c\`{a}dl\`{a}g).
For a (finite or infinite) closed interval $I\subseteq \R$, let $D(I)$ be the Skorokhod space on $I$, {i.e.}, the space of all c\`{a}dl\`{a}g functions defined on $I$. Let $\mathcal D_I$ be the Borel $\sigma$-field of the Skorokhod topology $J_1$. Denote by $\mathcal G_I=\mathcal M_1(D(I))$ the set of probability measures on $(D(I), \mathcal D_I)$. Our first step is to generalize the definition of compatibility to this setting, which follows in a natural way.

\begin{defn}
For a closed interval $I\subseteq \R$, we say $(Q_i)_{i\in\mathcal J}\subseteq \mathcal M_1$ and $(G_i)_{i\in\mathcal J}\subseteq \mathcal G_I$ are \emph{compatible} if there exists a c\`{a}dl\`{a}g stochastic process defined on $(\Omega,\mathcal A)$, denoted by $X=\{X(t)\}_{t\in I}$, such that
for each $i\in \mathcal J$, the distribution of $X$ under $Q_i$ is $ G_i$.
\end{defn}

The following is a parallel result to Theorem \ref{th:newstuff}, which shares the same proof.
\begin{prop}\label{prop:prop1}
Let $I\subseteq \R$ be a closed interval, $(Q_i)_{i\in\mathcal J}\subseteq \mathcal M_1$ and $(G_i)_{i\in\mathcal J}\subseteq \mathcal G_I$.
 A stochastic process $X$ has distribution $G_i$ under $Q_i$ for $i\in\mathcal J$ if and only if for all $Q\in \mathcal M_1$ dominating $(Q_i)_{i\in\mathcal J}$,  $G=Q \circ X^{-1}$ dominates $(G_i)_{i\in\mathcal J}$, and for all $i\in\mathcal J$,
$$ \frac{\d G_i}{\d G}(X) =   \E^Q\left[\left.\frac{\d Q_i}{\d {Q}}\right|\sigma(X)\right].$$
\end{prop}

Then we have, parallel to Theorem \ref{th:meas.variable}:

\begin{theorem}\label{th:stopro}
  Suppose that $(Q_1,\dots,Q_n)\in \mathcal M_1^n$ is conditionally atomless,  $I\subseteq \R$ is a closed interval, and $(G_1,\dots,G_n)\in\mathcal G_I^n$.
$(Q_1,\dots,Q_n)$ and $(G_1,\dots,G_n)$ are compatible if and only if  $(G_1,\dots,G_n)\lv (Q_1,\dots,Q_n)$.
\end{theorem}

\begin{proof}
In the proof of Lemma \ref{lem}, no structure of the real line has been used. As a result,  Lemma \ref{lem} can be directly generalized to the case of stochastic processes, with $(G_1,\dots,G_n)\in{\mathcal G}_I^n$ replacing $(F_1,\dots,F_n)\in{\mathcal F}^n$. For the other direction, the proof is similar to that of Theorem \ref{th:meas.variable}.  The only difference is that $(\mathbb R, {\mathcal B}(\mathbb R))$ is replaced by $(D(I), \mathcal D_I)$. A careful check of the proofs of Theorem \ref{th:meas.variable} and of Lemma \ref{lem:transplant} shows, however, that they only rely on the completely metrizable structure of $(\mathbb R, {\mathcal B}(\mathbb R))$ to guarantee the existence and uniqueness of the limit of the constructed sequence of random variables. Since $(D(I), \mathcal D_I)$ is also completely metrizable, the proofs naturally extend to the case of stochastic processes. More precisely, order the rational numbers in $I$ as ${\mathbb Q}\cap I=\{t_1, t_2,\dots\}$. Then we replace the refining partition of the real line $\{[h2^{-\ell}, (h+1)2^{-\ell}), ~h\in\mathbb Z\}_{\ell=0,1,\dots}$ with the refining partition of $D(I)$: $\{X(t_i)\in[h_{\ell,i}2^{-\ell+i}, (h_{\ell,i}+1)2^{-\ell+i}),~ i=1,\dots,\ell,~ h_{\ell,i}\in \mathbb Z\}_{\ell=1,2,\dots}$. The rest follows in the same way as in the proofs of Theorem \ref{th:meas.variable} and of Lemma \ref{lem:transplant}.
\end{proof}

\begin{rem}
The proof of Theorem \ref{th:stopro} sheds light upon a more general result, where $(G_1,\dots,G_n)$ are probability measures defined on a Polish space $Y$ equipped with the Borel $\sigma$-field. In particular, let $\{y_i\}_{i=1, 2, \dots}$ be a dense subset of $Y$, then the sequence of partitions
$$
\left\{\{y\in Y: d(y,y_i)\in[h_{\ell,i}2^{-\ell+i}, (h_{\ell,i}+1)2^{-\ell+i})\}, ~i=1,\dots,\ell,~ h_{\ell,i}=0,1,\dots\right\},
$$
$\ell=0,1,\dots$, can be used to replace $\{[h2^{-\ell}, (h+1)2^{-\ell}), ~h\in\mathbb Z\}_{\ell=0,1,\dots}$ in the proof of Lemma \ref{lem:transplant}. The rest follows exactly in the same way as in that proof. Consequently, a general version of Theorem \ref{th:meas.variable} can be stated using the setting of a Polish space instead of $\mathbb R$. However, due to the lack of a natural order and metric as in $\mathbb R$, a rigorous proof directly for the general case of a Polish space would be notationally heavy and also less intuitive for the readers who are not familiar with Polish spaces. As such, we present Theorem \ref{th:meas.variable} under the setting of $\mathbb R$, which is also the focus of this paper, and use this remark for a discussion for the general setting, after seeing the proof of Theorem \ref{th:stopro}.
\end{rem}

\subsection{Relation to the Girsanov Theorem}\label{sec:girsa}

In this section we investigate how much  the drift of a Brownian motion may vary under a change of measure as in the classic Girsanov Theorem. We keep in mind that, the distribution of a Brownian motion (with respect to its natural filtration) with a deterministic drift process only depends on this drift. On the other hand, Brownian motions with stochastic drift processes are not identified by the distribution of the drift processes. Due to this reason, we consider only Brownian motions with deterministic drift processes here.

Throughout this section,
let $P\in\mathcal M_1$ and $B=\{B_t\}_{t\in [0,T]}$ be a $P$-standard Brownian motion. Furthermore,
for a $[0,T]$-square integrable deterministic process $ \theta=\{\theta_t\}_{t\in [0,T]}$,
define
$$ \frac{\d Q_\theta}{\d P}=e^{\int_0^T \theta_t\d  B_t-  \frac{1}{2}\int_0^T\theta^2_t \d t},$$
and let $G_\theta$ be the distribution measure of a Brownian motion with drift process $\theta$.
The Girsanov Theorem says that  $B$ is a Brownian motion with drift process $\theta$ and volatility $1$ under $Q_\theta$ (certainly, this statement is also true for adapted drift processes).  Thus, $(P,Q_\theta)$ and $(G_0,G_\theta)$ are compatible.
 It is clear that distribution measures of Brownian motions with different non-random volatility terms are mutually singular, and hence they are not compatible with $(P,Q_\theta)$.
A next question is whether there exists a $P$-standard Brownian motion which has a deterministic drift process $\mu=\{\mu_t\}_{t\in [0,T]}$ under $Q_\theta$.
We are interested in the values of $\mu$ such that $(G_0,G_{\mu})$ and $(P,Q_\theta)$ above are compatible. Here we do not assume that $(P,Q_\theta)$ is conditionally atomless, which means that there might not be any random source other than $B$.

\begin{theorem}\label{thm:girsa}
Suppose that  the deterministic processes $\theta=\{\theta_t\}_{t\in [0,T]}$ and $\mu=\{\mu_t\}_{t\in [0,T]}$ are $[0,T]$-square integrable, and $\mu_t\ne0$ almost everywhere on $[0,T]$.  $(P,Q_\theta)$ and  $(G_0,G_{\mu})$ are compatible if and only if
$$\int_0^T \mu_t^2\d t\le \int_0^T \theta_t^2\d t.$$
\end{theorem}
\begin{proof}
(i) \underline{Necessity.}
 By the Girsanov Theorem, we know that $(G_0,G_\mu)$ and $(P,Q_\mu)$ are compatible. Using  Proposition \ref{prop:prop1} for $n=2$, we have
$$\frac{\d G_\mu}{\d G_0}(B)=\E\left[\left.\frac{\d Q_\mu}{\d P}\right|\sigma(B)\right] =e^{\int_0^T \mu_t\d  B_t-  \frac{1}{2}\int_0^T\mu^2_t \d t}.$$

Suppose that $(P,Q_\theta)$ and  $(G_0,G_{\mu})$ are compatible. Note that
$$e^{\int_0^T \mu_t\d  B_t-  \frac{1}{2}\int_0^T\mu^2_t \d t}\Big|_P\laweq\frac{\d G_\mu}{\d G_0}(B)\Big|_{P} \laweq\frac{\d G_\mu}{\d G_0}\Big|_{G_0}.$$

By Theorem \ref{th:stopro}, we have $$e^{\int_0^T \mu_t\d  B_t-  \frac{1}{2}\int_0^T\mu^2_t \d t} \Big|_P \lcx \frac{\d Q_\theta}{\d P}\Big|_P \laweq e^{\int_0^T \theta_t\d  B_t-  \frac{1}{2}\int_0^T\theta^2_t \d t}\Big|_P.$$

Applying the convex function $x\mapsto x^2$, we have
$$e^{\int_0^T \mu_t^2\d t}=\E[(e^{\int_0^T \mu_t\d  B_t-  \frac{1}{2}\int_0^T\mu^2_t \d t})^2]\le \E[(e^{\int_0^T \theta_t\d  B_t-  \frac{1}{2}\int_0^T\theta^2_t \d t})^2]=e^{\int_0^T \theta_t^2\d t}$$
and hence $\int_0^T \mu_t^2\d t\le \int_0^T \theta_t^2\d t$.

(ii) \underline{Sufficiency.}
Suppose $\int_0^T \mu_t^2\d t\le \int_0^T \theta_t^2\d t$. Define a deterministic process $\alpha=\{\alpha_t\}_{t\in [0,T]}$ by
$$
\alpha_t=\inf\left\{r\ge 0: \int_0^r \theta_s^2\d s =\int_0^t \mu_s^2\d s\right\}.
$$

It is easy to see that $\alpha_t$ is  strictly increasing  in $t$,   $\alpha_T\le T$, and
furthermore,
\begin{equation}
\label{eq:girsa2} \theta^2_{\alpha_t} \d \alpha_t =\mu_t^2 \d t.
\end{equation}

Let a stochastic process $\hat B=\{\hat B_t\}_{t\in [0,T]}$ be given by $\d \hat B_t=\d B_t - \theta_t\d  t$. By the Girsanov Theorem,
$\hat B$ is a $Q_\theta$-standard Brownian motion.
Define $$W_t= \int_0^{t} \beta_{\alpha_s} \d B_{\alpha_s},~~t\in[0,T],$$
where $\beta=\{\beta_s\}_{s\in [0,\alpha_T]}$ is given by $\beta_{\alpha_t}=\frac{\theta_{\alpha_t}}{\mu_t}$, $t\in [0,T]$.
$W=\{W_t\}_{t\in [0,T]}$ is clearly a Gaussian process,  $\E^P[W_t]=0$, and
$$\E^P[W_tW_s]=\E^P[W_s^2]= \int_0^s \frac{\theta^2_{\alpha_u}}{\mu^2_u}\d \alpha_u= s,~~0\le s<t\le T.$$

Therefore, $W$ is a $P$-standard Brownian motion. Furthermore,  for $t\in [0,T],$
\begin{align*}
W_t= \int_0^{t} \beta_{\alpha_s} \d B_{\alpha_s}                 &= \int_0^{t} \beta_{\alpha_s} (\d \hat B_{\alpha_s}+\theta_{\alpha_s} \d \alpha_s)\\
 &=\int_0^{t} \beta_{\alpha_s} \d \hat B_{\alpha_s}+  \int_0^{t} \beta_{\alpha_s} \theta_{\alpha_s} \d \alpha_s\\
  &=\int_0^{t} \beta_{\alpha_s} \d \hat B_{\alpha_s}+  \int_0^{t}\mu_s\d s,
  \end{align*}
  where the last equality is due to \eqref{eq:girsa2}.
As $\int_0^{t} \beta_{\alpha_s} \d \hat B_{\alpha_s}$ defines a $Q_\theta$-standard Brownian motion, we conclude that $W$ has distribution $G_\mu$ under $Q_\theta$, and hence $(P,Q_\theta)$ and  $(G_0,G_{\mu})$ are compatible.
\end{proof}

We list Theorem \ref{thm:girsa} for the case of a constant drift term below, and look more closely at the construction of the  desired stochastic process.
\begin{cor}\label{prop:girsa}
Let $\theta_t=a$ and $\mu_t=b$, $t\in [0,T]$, where $a, b$ are two constants, and $b\ne 0$.
 $(P,Q_\theta)$ and  $(G_0,G_{\mu})$ are compatible if and only if $b^2 \le a^2$.
\end{cor}

If $b^2 \le a^2$, the process which  has distribution $G_0$ under $P$ and distribution $G_\mu$ under $Q_\theta$ can be written in a simple explicit form.
Let $$W_t=\frac{a}{b}B_{(\frac{b}{a})^2t},~~t\in[0,T].$$
It is clear that $W=\{W_t\}_{t\in [0,T]}$ is a $P$-Brownian motion. Furthermore,
$$W_t=\frac{a}{b}B_{(\frac{b}{a})^2t}=\frac{a}{b}\left(\hat B_{(\frac{b}{a})^2t}+ a \frac{b^2}{a^2} t\right)= \frac{a}{b}\hat B_{(\frac{b}{a})^2t}+ b t,~~t\in [0,T].$$

In this example, it is clear that $0<b^2\le a^2$ is essential; otherwise $W$ will not be well-defined.

\section{Application to a portfolio optimization problem}\label{sec:5}


 Let $P, Q, R$ be three probability measures defined on $(\Omega,\mathcal A)$, and $F, G$ be two probability distributions on $\mathbb R$.
 In this section, we investigate the optimization problem of the type
 \begin{equation}\label{eq:optimization}
 \begin{array}{c}
 \min_{H\in \mathcal F} H([a,\infty))\\
 \text{s.t. } (P, Q, R) \text{ is compatible with } (F, G, H).
 \end{array}
\end{equation}

Problem \eqref{eq:optimization} is motivated by 
portfolio selection under multiple constraints.
In a classic complete-market portfolio selection problem, an investor optimizes an objective function under the physical measure $P$ (e.g.~expected utility) subject to a budget constraint which is evaluated under a risk-neutral  measure $Q$.
For this problem, one obtains an optimal position with loss random variable $\xi^*$, and we denote by $F$ (resp.~$G$) the distribution of $\xi^*$ under $P$ (resp.~$Q$).
The optimal position $\xi^*$ may not be unique but its distributions under $P$ and $Q$ are typically unique (see e.g.~Chapter 3 of \cite{FS16}).
We assume, in addition, that there is a regulatory requirement set by a regulator using a measure $R$ which may not be the same as $P$ due to extensive usage of stress-testing in calculating regulatory capital (see e.g.~\cite{CF17}).
A typical regulatory requirement is using the Value-at-Risk under stressed scenarios (see e.g.~\cite{MFE15}), that is, a loss random variable $\xi$ has to satisfy $R(\xi\ge a)\le p_0$ where $a \in \R$ is the capital level of the investor and $p_0$ is a pre-specified probability level.
The investor needs to determine whether a position with her desired distributions under $P$ and $Q$ can satisfy this constraint. That is,
 to determine the existence of a random variable $\xi$, such that
 $$
\xi|_P\sim F, \quad \xi|_Q\sim G, \text{ and } R(\xi\geq a)\leq p_0.
 $$

 Using the framework of this paper, this is  to determine the existence of a probability distribution $H$, such that $(P, Q, R)$ is compatible with $(F, G, H)$, and $H([a,\infty))\leq p_0$.
 It is obvious that the optimization problem \eqref{eq:optimization} directly addresses the above issue.

To study \eqref{eq:optimization}, we assume that $(P, Q, R)$ are conditionally atomless and $Q, R\ll P$. By Theorem \ref{th:meas.variable}, the compatibility in \eqref{eq:optimization} is equivalent to the heterogeneity order $(F, G, H)\lv (P, Q, R)$. An application of Lemma \ref{lem:lem33} (iii) shows this is equivalent to
$$
\left.\left(\frac{\d F}{\d F'}, \frac{\d G}{\d F'}, \frac{\d H}{\d F'}\right)\right|_{F'} \lcx \left.\left(1, \frac{\d Q}{\d P}, \frac{\d R}{\d P}\right)\right|_P
$$
for some $F'$ such that $F, G, H\ll F'$. As $\left.\frac{\d F}{\d F'}\right|_{F'}\lcx 1$, $F'$ must be the same as $F$. Hence, an equivalent condition is  
$$
\left.\left(\frac{\d G}{\d F}, \frac{\d H}{\d F}\right)\right|_F \lcx \left.\left(\frac{\d Q}{\d P}, \frac{\d R}{\d P}\right)\right|_P.
$$

For simplicity, we assume $\frac{\d G}{\d F}$ is $F$-a.e.~injective. 
As such, we can write
$ H([a,\infty))=H(\{t\in \R: \frac{\d G}{\d F}(t)\in D_a\})$ for some measurable set $D_a\subseteq \mathbb R$.
By Lemma \ref{lem:lem2}, the compatibility holds if and only if there exists some probability space $(\Omega', \mathcal A', P')$ and random variables $X', Y', Z', W'$ defined on that space, such that
$$
(X', W')|_{P'}\laweq \left.\left(\frac{\d G}{\d F}, \frac{\d H}{\d F}\right)\right|_{F},~ (Y', Z')|_{P'}\laweq \left.\left(\frac{\d Q}{\d P},\frac{\d R}{\d P}\right)\right|_P,
$$
and
$$
(X', W')= \E^{P'}[(Y', Z')|X'].
$$

 The relation $\E^{P'}[(Y', Z')|X', W']=\E^{P'}[(Y', Z')|X']$ is used above, which is guaranteed by $(X', W')|_{P'}\laweq (\frac{\d G}{\d F}, \frac{\d H}{\d F})|_{F}$ and that $\frac{\d G}{\d F}$ is injective.
We take $(\Omega', \mathcal A', P')$  as fixed from now on, since only distributions matter in our optimization problem.
We have
\begin{align*}
H([a,\infty))&=\E^F\left[\frac{\d H}{\d F} \id_{\{\frac{\d G}{\d F} \in D_a\}}\right]=\E^{P'}[W' \id_{\{X'\in D_a\}}]\\
&=\E^{P'}[Z' \id_{\{X'\in D_a\}}]=\E^{P'}[\E^{P'}[Z' \id_{\{X'\in D_a\}}|Y']].
\end{align*}

Hence, we relax the reliance of $W'$ in the  optimization problem, and \eqref{eq:optimization} can be rewritten as
\begin{equation}\label{eq:altopt}
\min_{(X',Y',Z')} \E^{P'}[Z' \id_{\{X'\in D_a\}}],
\end{equation}
where the minimum is taken subject to the constraints
\begin{equation*} 
  X'|_{P'}\laweq \left.\frac{\d G}{\d F}\right|_{F},~~ (Y', Z') |_{P'}\laweq  \left.\left(\frac{\d Q}{\d P},\frac{\d R}{\d P}\right)\right|_P,  \mbox{~~and~~}
 X'= \E^{P'}[Y'|X'].
 \end{equation*}

Under $P'$, given the joint distribution of $X'$ and $Y'$, the conditional distributions $X'|Y'=y$ and $Z'|Y'=y$ are both fixed for $P'$-almost every $y$. Hence, by the Hardy-Littlewood inequality (in the form of Remark 3.25 of \cite{R13}), the sub-problem, for fixed $(X',Y')$,
$$
\min_{Z'\in K_y}\E^{P'}[Z' \id_{\{X'\in D_a\}}|Y'=y],
$$
where $K_y$ is the set of all random variables $Z'$ satisfying $ (Z'|Y'=y)\big|_{P'}\laweq  (\frac{\d R}{\d P}|\frac{\d Q}{\d P} =y) \big|_{P}$, has a simple solution such that $Z'$ given $Y'=y$ and $\id_{\{X\in D_a\}}$ given $Y'=y$ are counter-monotonic. Consequently, we have
$$
\min_{Z'\in K_y}\E^{P'}[Z' \id_{\{X'\in D_a\}}|Y'=y]=\int_0^{p_{X',Y'}(y)}f(x|y)\d x,
$$
where $f(\cdot|y)$ is the left-quantile of the distribution function of $\frac{\d R}{\d P} $ given $ \frac{\d Q}{\d P} =y$ under $P$, and
\begin{equation*}
p_{X',Y'}(y)=P'(X'\in D_a|Y'=y).
\end{equation*}

Define a function $\Phi$ on $(\mathcal L^0(\Omega';\mathbb R))^2$ by
$$\Phi(X',Y') = \E^{P'}\left[
 \int_0^{p_{X',Y'}(Y')}f (x|Y')\d x \right]=  \int_0^\infty
 \int_0^{p_{X',Y'}(y)}f (x|y)\d x \d F_Y(y),
$$
where $F_Y$ is the distribution of $\frac{\d Q}{\d P}\big |_P$.
Clearly, $\Phi$ is determined by the joint distribution of $(X',Y')$ under $P'$.
By this argument, we relax the reliance of $Z'$ in the optimization problem \eqref{eq:altopt}.
To summarize, the results in this paper allow us to transform the original optimization problem \eqref{eq:optimization} into
\begin{equation}\label{eq:optequiv}
\min_{(X,Y)\in K} \Phi(X,Y)
\end{equation}
where $K$ is the set of all random variables $(X,Y)\in(\mathcal L^0(\Omega';\mathbb R))^2$ satisfying
\begin{equation} \label{eq:cons3}
X|_{P'}\laweq \left.\frac{\d G}{\d F}\right|_{F},~~Y|_{P'}\laweq \left.\frac{\d Q}{\d P}\right|_{P}, \mbox{~~and~~} \E^{P'}[Y|X]=X.
\end{equation}
\begin{rem}
Problem \eqref{eq:optequiv} can be seen as a generalized martingale mass transportation problem (e.g.~\cite{BHP13}).
 In a classic two-period martingale mass transportation  problem, the objective is to minimize $\E^\p [\phi(X,Y)]$ for some cost function $\phi: \R^2\to \R$ over $(X,Y)$
 where the distributions of $X$ and $Y$ under some measure $\p$  are known, and $\E^\p [Y|X]=X$.
 Note that our constraints \eqref{eq:cons3} are the same as in the classic problem. The only difference between \eqref{eq:optequiv} and the classic problem is that our objective $\Phi$
does not have the form of an expected value of $\phi(X,Y)$. 
Rather, $\Phi$ is determined by the joint distribution of $(X,Y)$.
 Hence, $\Phi$ can be seen as a generalized cost functional in a mass transportation problem.
\end{rem}


Meanwhile, a lower bound for the optimal value of (\ref{eq:optequiv}) can be obtained by considering an optimization problem with a weaker constraint:
\begin{equation}\label{eq:weakopt}
\min_{Y\in K'}\Phi(X,Y),
\end{equation}
where $K'$ is the set of all random variables $Y$ satisfying $Y|_{P'}\laweq \frac{\d Q}{\d P}\big|_{P}$, and
$$
\E^{P'}[Y \id_{\{X\in D_a\}}]=\E^{P'}[X \id_{\{X\in D_a\}}].
$$
Note that since only the joint distribution of $X$ and $Y$ matters, here we take $X$ as given and reduce the problem to an optimization solely over $Y$. 
  Denote by $Y^*$ an optimal solution of (\ref{eq:weakopt}), and let ${p}^*=p_{X,Y^*}$. Then for any two points $y_1$, $y_2$ and $\lambda\in (0,1)$ satisfying $p^*(y_1), p^*(y_2), p^*(\lambda y_1+(1-\lambda) y_2)\in(0,1)$, 
  a variational argument leads to
the first order condition 
\begin{equation}\label{eq:firstorder}
f(p^*(\lambda y_1+(1-\lambda)y_2)|\lambda y_1+(1-\lambda)y_2)=\lambda f(p^*(y_1)|y_1)+(1-\lambda)f(p^*(y_2)|y_2),
\end{equation}
which implies that $f(p^*(y)|y)$ must be linear in $y$ when $p^*(y)$ is between 0 and 1. Combining this with the constraints
$$
\E^{P'}[p^*(Y^*)]=P'(X\in D_a)=F([a,\infty))
$$
and
$$
\E^{P'}[Y^*p^*(Y^*)]=\E^{P'}[Y^* \id_{\{X\in D_a\}}]=\E^{P'}[X \id_{\{X\in D_a\}}]
$$
generically gives a unique solution, which is a local minimum by checking the second order condition. 
Note that similar to (\ref{eq:altopt}), (\ref{eq:weakopt}) can be rewritten as $\min \E^{P'}[Z \id_{\{X\in D_a\}}]$, where the minimum is taken over all the $(Y, Z)$ such that $(Y, Z)|_{P'}\laweq (\frac{\d Q}{\d P},\frac{\d R}{\d P})|_{P}$ and $\E^{P'}[Y \id_{\{X\in D_a\}}]= \E^{P'}[X \id_{\{X\in D_a\}}]$.
As the objective $\E^{P'}[Z \id_{\{X\in D_a\}}]$ is linear and the feasible region is convex 
(with respect to mixture),
 the local minimum must also be the global minimum for the optimization problem (\ref{eq:weakopt}), providing a lower bound for the optimal value in (\ref{eq:optequiv}).

In some special cases, the above lower bound can be analytically calculated, and one can construct random variables satisfying the original constraints that attain this bound. As a result, (\ref{eq:optequiv}) and (\ref{eq:weakopt}) have the same optimal value, and the original problem \eqref{eq:optimization} is completely solved.
  We give one simple example. Let $F$ and $G$ be supported on $[0,1]$, $a\geq \frac{1}{2}$, and $\frac{\d G}{\d F}$ follows a symmetric triangular distribution under $F$ and is decreasing: $\frac{\d G}{\d F}(t)=2-\sqrt{2t}$ for $t\in[0,1/2)$ and $\frac{\d G}{\d F}(t)=\sqrt{2-2t}$ for $t\in[1/2,1]$. $\frac{\d Q}{\d P}$ follows a uniform distribution on $[0,2]$, and the conditional distribution of $\frac{\d R}{\d P}$ given $\frac{\d Q}{\d P}$ is uniform with linear bounds: $\frac{\d R}{\d P}\big|\frac{\d Q}{\d P}=y\sim \text{Unif}([cy-b, cy+b])$ for some constants $b$ and $c$. In this case we construct $(X,Y)$ such that $X|_{P'}\laweq \frac{\d G}{\d F}\big|_{F}$, and
$$
Y|X=x\sim\begin{cases}
\text{Unif}([0,2x]) & x\in[0,1)\\
\text{Unif}([2x-2,2]) & x\in[1,2].
\end{cases}
$$
$p_{X,Y}(y)$ can be derived and then it can be verified that the first order condition (\ref{eq:firstorder}) is met. Consequently, the dependence given by $(X, Y)$ is indeed optimal for problem (\ref{eq:optequiv}), and the corresponding optimal value can be calculated. We omit the detail as the rest is purely computational.

\subsection*{Acknowledgements}
The authors are grateful to the Editor, the Associate Editor,  two referees, Michel Baes, Fabio Bellini, Paul Embrechts, Fabio Maccheroni, Tiantian Mao, Alfred M\"uller, Marcel Nutz, Jan Obloj, Sidney Resnick, Ludger R\"uschendorf, Alexander Schied and Xiaolu Tan for various helpful suggestions and discussions on an earlier version of the paper.
J.~Shen  acknowledges financial support from the China Scholarship Council.
Y.~Shen and R.~Wang acknowledge financial support by the Natural Sciences and Engineering Research Council (NSERC 2014-04840, RGPIN-2018-03823, RGPAS-2018-522590) of Canada.
R.~Wang is also grateful to FIM at ETH Zurich for supporting his visit in 2017, during which part of this paper was written.

%
%

%
%
%

\appendix\normalsize
\allowdisplaybreaks
\section{Appendix}

\subsection{Details in Example \ref{ex:ct1}}
\label{app:a1}

Note that $\frac{\d Q_1}{\d Q_2}$  is uniform on $[0,2]$ under $Q_2=\lambda$, and $\frac{\d F_1}{\d F_2}$ is also uniform on $[0,2]$ under $F_2=\lambda$.
Thus,
$$\left.\left(\frac{\d F_1}{\d \lambda},\frac{\d F_2}{\d \lambda}\right)\right|_\lambda \laweq \left. \left(\frac{\d Q_1}{\d \lambda},\frac{\d Q_2}{\d \lambda}\right)\right|_\lambda.  $$
Therefore, $(F_1,F_2)\lv (Q_1,Q_2)$.

Next, we will see that $(Q_1,Q_2)$ and $(F_1,F_2)$ are not compatible.
Suppose for the purpose of contradiction that $(Q_1,Q_2)$ and $(F_1,F_2)$ are compatible. By Theorem \ref{th:newstuff}, there exists a random variable  $X$ in $(\Omega,\mathcal A)$ with a uniform distribution on $[0,1]$  under $Q_2=\lambda$ such that
$$\frac{\d F_1}{\d \lambda}(X)=\E^{\lambda}\left[\left.\frac{\d Q_1}{\d \lambda}\right| X\right].  $$
In addition,
$$\left.\frac{\d F_1}{\d \lambda}(X)\right|_{\lambda}\laweq\left. \frac{\d Q_1}{\d \lambda}\right|_{\lambda},$$ and therefore,
$$\frac{\d F_1}{\d \lambda}(X)= \frac{\d Q_1}{\d \lambda},~~\mbox{$\lambda$-almost surely.}  $$

From the definition of $F_1$ and $Q_1$, we have, for $\lambda$-almost surely $t\in [0,1]$,
$|4X(t)-2|=2t.$
It follows that $X(t)=(t+1)/2$ or $X(t)=(1-t)/2$ for all $t\in [0,1]$.
Write $$A=\left\{t\in [0,1]: X(t)=\frac{t+1}{2}\right\},~ B=\left\{t\in [0,1]: X(t)=\frac{1-t}{2}\right\}$$
  and
  $$C=\left\{\frac{1-t}{2}:t\in A\right\}.$$

  As $X$ is $\mathcal B([0,1])$-measurable and has distribution $F_2$ under $\lambda$, we have $A,B \in \mathcal B([0,1])$ and $\lambda(A)=\lambda(B)=1/2$.
Note that $\lambda(C)=1/4$; however $\lambda (C\cap X(A\cup B))=0$, contradicting the fact that $X$ has  a uniform distribution on $[0,1]$ under $\lambda$.

\subsection{Proof of Theorem \ref{th:almost:compatible}}\label{app:a2}
\begin{proof}
\underline{Necessity.} Assume that $(Q_1,\dots,Q_n)$ and $(F_1,\dots,F_n)$ are almost compatible. This means that for any $\epsilon>0$, there exists $(F_{1,\epsilon},\dots,F_{n,\epsilon})$ such that $D_{\mathrm{KL}}(F_{i,\epsilon}\Vert F_i)<\epsilon$ for $i=1,\dots,n$, and $(Q_1,\dots,Q_n)$ is compatible with $(F_{1,\epsilon},\dots,F_{n,\epsilon})$. Define probability measures
 $$
 F_\epsilon=\frac{1}{n}(F_{1,\epsilon}+\dots+F_{n,\epsilon}),
 $$
 $$
 F=\frac{1}{n}(F_{1}+\dots+F_{n})
 $$
 and
 $$
 Q=\frac{1}{n}(Q_1+\dots+Q_n).
 $$
 Note that the distribution of $X_\epsilon$ under $Q$ is $F_\epsilon$, where $X_\epsilon$ is the random variable defining the compatibility between $(Q_1,\dots,Q_n)$ and $(F_{1,\epsilon},\dots,F_{n,\epsilon})$. Moreover, for $i=1,\dots,n$, we have $F_{i,\epsilon}\ll F_\epsilon$, $Q_i\ll Q$, $\d F_{i,\epsilon}/\d F_\epsilon\leq n$ and $\d Q_i/\d Q\leq n$.
 For $\epsilon>0$, by Lemma \ref{lem},
$$
\left.\left(\frac{\d F_{1,\epsilon}}{\d F_\epsilon},\dots, \frac{\d F_{n,\epsilon}}{\d F_\epsilon} \right)\right|_{F_\epsilon}\lcx \left. \left(\frac{\d Q_1}{\d Q},\dots,\frac{\d Q_n}{\d {Q}}\right)\right|_{Q}.
$$
As a result, for any convex function $f:{\mathbb R}^n\to \mathbb R$,
$$
\E^{F_\epsilon}\left[f\left(\frac{\d F_{1,\epsilon}}{\d F_\epsilon},\dots, \frac{\d F_{n,\epsilon}}{\d F_\epsilon}\right) \right]\leq \E^Q\left[f\left(\frac{\d Q_1}{\d Q},\dots,\frac{\d Q_n}{\d {Q}}\right)\right].
$$

For $i=1,\dots,n$,
\begin{equation}\label{eq:converge}
\frac{\d F_{i,\epsilon}}{\d F_\epsilon}=\frac{\d F_i}{\d F}\frac{\d F_{i,\epsilon}/\d F_i}{\d F_\epsilon/\d F}.
\end{equation}
Since $D_{\mathrm{KL}}(F_{i,\epsilon}\Vert F_i)$ converges to 0, by Pinsker's inequality,
 $F_{i,\epsilon}$ converges to $F_i$ in total variation, which is equivalent to $\d F_{i,\epsilon}/\d F_i$ converging in $L^1|_{F_i}$ to $1$. Hence for any sequence $\epsilon_m\downarrow 0$,
 there exists a subsequence, which we still denote as $\epsilon_m\downarrow 0$ by a slight abuse of notation, such that $\d F_{i,\epsilon_m}/\d F_i$ converge to $1$ $F_i$-almost surely. It is easy to check that we have $\d F_{\epsilon_m}/\d F$ converge to 1 as well. (\ref{eq:converge}) then implies that
 \begin{equation}\label{eq:almost:sure:partial}
 \frac{\d F_{i,\epsilon_m}}{\d F_{\epsilon_m}}\to \frac{\d F_i}{\d F} \quad F_i\text{-almost surely}.
 \end{equation}

 On any set $B\in \mathcal B(\mathbb R)$ such that $F_i(B)=0$ but $F(B)>0$, suppose $\d F_{i,\epsilon}/\d F_{\epsilon}$ does not converge to $\d F_i/\d F=0$ in probability under $F|_B$, the measure $F$ restricted on $B$. Then there exists $\delta>0$ and a subsequence of $\epsilon_m$ (again denoted as $\epsilon_m$), such that $P^{F|_B}(\d F_{i,\epsilon_m}/\d F_{\epsilon_m}>\delta)\geq c$ for some constant $c>0$. Since $F_{\epsilon_m}$ converges to $F$ in total variation, for $m$ large enough, $P^{F_{\epsilon_m}|_B}(\d F_{i,\epsilon_m}/\d F_{\epsilon_m}>\delta)\geq c/2$. Hence $F_{i,\epsilon_m}(B)\geq \delta P^{F_{\epsilon_m}|_B}(\d F_{i,\epsilon_m}/\d F_{\epsilon_m}>\delta)\geq \frac{c\delta}{2}$, which contradicts the fact that $F_{i,\epsilon_m}$ converges to $F_i$ in total variation. We conclude that $\d F_{i,\epsilon}/\d F_{\epsilon}$ converge to $\d F_i/\d F=0$ in probability under $F$ on set $\{\d F_i/\d F=0\}$. Combining this result with (\ref{eq:almost:sure:partial}) and taking a further subsequence allows us to replace the $F_i$-almost sure convergence in (\ref{eq:almost:sure:partial}) by $F$-almost sure convergence.

 For any convex function $f:{\mathbb R}^n\to \mathbb R$,
$$
\E^{F_{\epsilon_m}}\left[f\left(\frac{\d F_{1,\epsilon_m}}{\d F_{\epsilon_m}},\dots, \frac{\d F_{n,\epsilon_m}}{\d F_{\epsilon_m}}\right) \right]=\int f\left(\frac{\d F_{1,\epsilon_m}}{\d F_{\epsilon_m}},\dots, \frac{\d F_{n,\epsilon_m}}{\d F_{\epsilon_m}}\right)\d F_{\epsilon_m}.
$$
Since $\frac{\d F_{i,\epsilon_m}}{\d F_{\epsilon_m}}\in[0,n]$, and $f$ is convex hence continuous, $|f(\frac{\d F_{1,\epsilon_m}}{\d F_{\epsilon_m}},\dots, \frac{\d F_{n,\epsilon_m}}{\d F_{\epsilon_m}})|$ is bounded. Let $b$ be an upper bound of it. Because $F_{\epsilon_m}$ converges in total variation to $F$, we have
\begin{multline}\label{eq:uniform:conv}
\left|\int f\left(\frac{\d F_{1,\epsilon_m}}{\d F_{\epsilon_m}},\dots, \frac{\d F_{n,\epsilon_m}}{\d F_{\epsilon_m}}\right)\d F_{\epsilon_m}-\int f\left(\frac{\d F_{1,\epsilon_m}}{\d F_{\epsilon_m}},\dots, \frac{\d F_{n,\epsilon_m}}{\d F_{\epsilon_m}}\right)\d F\right|\\
\leq 2b\delta(F_{\epsilon_m},F)\to 0
\end{multline}
uniformly, where $\delta(\cdot, \cdot)$ is the total variation distance. Moreover, by dominated convergence, we have
\begin{equation}\label{eq:dominated:conv}
\int f\left(\frac{\d F_{1,\epsilon_m}}{\d F_{\epsilon_m}},\dots, \frac{\d F_{n,\epsilon_m}}{\d F_{\epsilon_m}}\right)\d F\to \int f\left(\frac{\d F_{1}}{\d F},\dots, \frac{\d F_{n}}{\d F}\right)\d F.
\end{equation}

(\ref{eq:uniform:conv}) and (\ref{eq:dominated:conv}) together show that
\begin{align*}
\E^{F}\left[f\left(\frac{\d F_{1}}{\d F},\dots, \frac{\d F_{n}}{\d F}\right) \right]&=\lim_{m\to\infty}\E^{F_{\epsilon_m}}\left[f\left(\frac{\d F_{1,\epsilon_m}}{\d F_{\epsilon_m}},\dots, \frac{\d F_{n,\epsilon_m}}{\d F_{\epsilon_m}}\right) \right]\\
&\leq \E^Q\left[f\left(\frac{\d Q_1}{\d Q},\dots,\frac{\d Q_n}{\d {Q}}\right)\right].
\end{align*}

 \underline{Sufficiency.} Assume that $(F_1,\dots,F_n)\lv  (Q_1,\dots,Q_n)$. By Lemma \ref{lem:lem33}, this means that
  $$\left.\left(\frac{\d F_1}{\d F},\dots, \frac{\d F_n}{\d F} \right)\right|_F \lcx \left. \left(\frac{\d Q_1}{\d Q},\dots,\frac{\d Q_n}{\d Q}\right)\right|_Q
$$ holds for $F=\frac{1}{n}\sum_{i=1}^n F_i$ and $Q=\frac{1}{n}\sum_{i=1}^n Q_i$.

By Lemma \ref{lem:lem2}, there exists a probability space $(\Omega', {\mathcal A}', Q')$ and random vectors $\mathbf Y'=(Y_1',\dots,Y_n'), \mathbf Z'=(Z_1',\dots,Z_n')$ defined on $(\Omega', {\mathcal A}', Q')$, such that
$$
(Y_1',\dots,Y_n')\laweq\left(\frac{\d Q_1}{\d Q},\dots,\frac{\d Q_n}{\d Q}\right)=:\mathbf Y=(Y_1,\dots,Y_n),
$$
$$
(Z_1',\dots,Z_n')\laweq\left(\frac{\d F_1}{\d F},\dots,\frac{\d F_n}{\d F}\right)=:\mathbf Z=(Z_1,\dots,Z_n),
$$
and
$$
\E^{Q'}[Y_i'|Z_i']=Z_i',\quad i=1,\dots,n.
$$

Given $m=0,1,\dots$, define random vector $\mathbf Y_m=(Y_{m,1},\dots,Y_{m,n})$ by
$$
Y_{m,i}=\left\{
\begin{array}{ll}
0 & \text{ if }Y_i=0\\
\exp(2^{-m}\lfloor2^m\log(Y_i)\rfloor)& \text{ otherwise}
\end{array}
\right.
$$
for $i=1,\dots,n$. Similarly we define $\mathbf Y'_m$, $\mathbf Z_m$ and $\mathbf Z'_m$ for $\mathbf Y'$, $\mathbf Z$ and $\mathbf Z'$, respectively. Note that
\begin{align*}
\E^{Q'}\left[Y_{m,i}'|Z_{m,i}'\right]&\in \left[\exp(-2^{-m})\E^{Q'}[Y_{i}'|Z_{m,i}'], \E^{Q'}[Y_{i}'|Z_{m,i}']\right]\\
&\subseteq \left[\exp(-2^{-m})Z_{m,i}', \exp(2^{-m})Z_{m,i}'\right]
\end{align*}
 for $i=1,\dots,n$.

Each of $Q_1,\dots, Q_n$ is atomless, and so is $Q$. As a result, we can divide $\Omega$ into disjoint sets $A^m_{k,j}$, where $k=(k_1,\dots,k_n)\in({\mathbb Z}\cup\{-\infty\})^n$ and $j=(j_1,\dots,j_n)\in({\mathbb Z}\cup\{-\infty\})^n$, such that $Y_{m,i}(\omega)=\exp (k_i2^{-m})$ for $\omega\in A^m_{k,j}$ and $i=1,\dots,n$,
 $$
 Q(A^m_{k,j})={Q'}(Y'_{m,i}=\exp(k_i2^{-m}), Z'_{m,i}=\exp(j_i2^{-m}), i=1,\dots,n).$$
  Here we follow the tradition that $\exp(-\infty)=0$ for ease of notation. Define random vector $\mathbf Z_m''$ on $(\Omega, {\mathcal A}, Q)$ by $Z_{m,i}''(\omega)=\exp(j_i2^{-m})$ for $\omega\in A^m_{k,j}$, then $(\mathbf Y_m,\mathbf Z''_m)|_Q\laweq (\mathbf Y_m',\mathbf Z_m')|_{Q'}$.

Let $I_d$ be the identity random variable on $(R, {\mathcal B}(\mathbb R))$. For $\ell=0,1,\dots$ and $h\in\mathbb Z$, denote by $\varphi^m_{\ell,h}(z)$ the conditional probability under $F$ of the event $I_d\in [h2^{-\ell}, (h+1)2^{-\ell})$ given $Z_m=z$:
 $$
 \varphi^m_{\ell,h}(z)=F(I_d\in [h2^{-\ell}, (h+1)2^{-\ell})|Z_m=z).
 $$
 Then for any $\ell=0,1,\dots$, $A^m_{k,j}$ can be further divided into disjoint subsets $A^m_{k,j,\ell,h}$, such that $Q(A^m_{k,j,\ell,h})$$=Q(A^m_{k,j})\varphi^m_{\ell,h}(\exp (j2^{-m}))$. Moreover, the partitions can be made such that $\{A^m_{k,j,\ell',h}\}_{h\in\mathbb Z}$ is a refinement of $\{A^m_{k,j,\ell,h}\}_{h\in\mathbb Z}$ for any $\ell'>\ell$ and any given $m,k,j$. Define $X_{m,\ell}(\omega)=h2^{-\ell}$ for $\omega\in A^m_{k,j,\ell,h}$, and $X_m=\lim_{\ell\to\infty}X_{m,\ell}$. The limit exists since it is easy to check that $X_{m,\ell}$ is increasing with respect to $\ell$. Note that $X_{m,\ell}$ is conditionally independent of $\mathbf Y_m$ given $\mathbf Z_m''$, hence $X_m$ is also conditionally independent of $\mathbf Y_m$ given $\mathbf Z''_m$.

 By construction, for any $A\in{\mathbb R}^n$, $\ell=0,1,\dots,$ and $h\in\mathbb Z$,
\begin{align}\label{eq:keep:value}
\begin{split}
&~Q(\mathbf Z''_m\in A, X_{m,\ell'}\in [h2^{-\ell},(h+1)2^{-\ell}))\\
=&~Q(\mathbf Z''_m\in A, X_{m,\ell}=h2^{-\ell})\\
=&~\sum_{\substack{k\\ j:\exp(j2^{-m})\in A}}Q(A^m_{k,j,\ell,h})\\
=&~\sum_{\substack{k \\ j:\exp(j2^{-m})\in A}}Q(A^m_{k,j})\varphi^m_{\ell,h}(\exp(j2^{-m}))\\
=&~\sum_{j:\exp(j2^{-m})\in A}Q(\mathbf Z''_m=\exp(j2^{-m}))\varphi^m_{\ell,h}(\exp(j2^{-m}))\\
=&~\sum_{j:\exp(j2^{-m})\in A}F(\mathbf Z_m=\exp(j2^{-m}))F([h2^{-\ell},(h+1)2^{-\ell})|\mathbf Z_m=\exp(j2^{-m}))\\
=&~F(\mathbf Z_m^{-1}(A)\cap [h2^{-\ell},(h+1)2^{-\ell}))
\end{split}
\end{align}
for all $\ell'\geq \ell$. Thus, $\mathbf Z_m$, restricted on interval $[h2^{-\ell},(h+1)2^{-\ell})$, has the same distribution as $\mathbf Z_m''$, restricted on set $X^{-1}_{m,\ell'}([h2^{-\ell},(h+1)2^{-\ell}))$. Note that $X^{-1}_{m,\ell'}([h2^{-\ell},(h+1)2^{-\ell}))$ is the same set for any $\ell'\geq \ell$, hence $\mathbf Z_m$ restricted on interval $[h2^{-\ell},(h+1)2^{-\ell})$ also has the same distribution as $\mathbf Z_m''$ restricted on $X^{-1}_m([h2^{-\ell},(h+1)2^{-\ell}))$ for all $m=0,1,\dots$. Because the collection of sets $\{[h2^{-\ell},(h+1)2^{-\ell})\}_{h\in{\mathbb Z},\ell=0,1,\dots}$ forms a basis for ${\mathcal B}(\mathbb R)$, $\mathbf Z_m$ restricted on any Borel set $B$ has the same distribution as $\mathbf Z_m''$ restricted on $X_m^{-1}(B)$. Therefore we conclude that $\mathbf Z''_m=\mathbf Z_m\circ X_m$ $Q$-almost surely. Moreover, by taking $A={\mathbb R}^n$ in (\ref{eq:keep:value}), it follows that $Q(X_{m,\ell'}\in [h2^{-\ell},(h+1)2^{-\ell}))=F([h2^{-\ell},(h+1)2^{-\ell}))$ for all $\ell'\geq \ell$. A similar reasoning as above then shows that $F=Q\circ X_m^{-1}$.

For any $A\in {\mathcal B}$ and any $i=1,\dots,n$,

\begin{equation}\label{eq:Qiexact}
Q_i(X_m\in A)=\int_{X_m^{-1}(A)}Y_i\d Q.
\end{equation}
It is easy to see that
\begin{equation}
\int_{X_m^{-1}(A)}Y_{m,i}\d Q\leq \int_{X_m^{-1}(A)}Y_i\d Q\leq \exp(2^{-m})\int_{X_m^{-1}(A)}Y_{m,i}\d Q.
\end{equation}

Moreover,
\begin{align*}
&\int_{X_m^{-1}(A)}Y_{m,i}\d Q\\
&=\sum_{j}Q\left(X_m\in A\big|\mathbf Z''_m=e^{j2^{-m}}\right)\sum_k e^{k_i2^{-m}}Q\left(\mathbf Y_m=e^{k2^{-m}}, \mathbf Z''_m=e^{j2^{-m}}\right)\\
&=\sum_{j}Q\left(X_m\in A\big|\mathbf Z''_m=e^{j2^{-m}}\right)Q\left(\mathbf Z''_m=e^{j2^{-m}}\right)\E^Q[Y_{m,i}|\mathbf Z''_m=e^{j2^{-m}}]\\
&=\sum_{j}Q\left(X_m\in A\big|\mathbf Z''_m=e^{j2^{-m}}\right)Q\left(\mathbf Z''_m=e^{j2^{-m}}\right)\E^{Q'}[Y'_{m,i}|\mathbf Z'_m=e^{j2^{-m}}]\\
&\geq \sum_{j}Q\left(X_m\in A\big|\mathbf Z''_m=e^{j2^{-m}}\right)Q\left(\mathbf Z''_m=e^{j2^{-m}}\right)\exp\left(j_i2^{-m}-2^{-m}\right)\\
&=\sum_{j}F\left(A\big |\mathbf Z_m=e^{j2^{-m}}\right)F\left(\mathbf Z_m=e^{j2^{-m}}\right)\exp\left(j_i2^{-m}-2^{-m}\right)\\
&\geq \exp\left(-2^{-m}\right)\sum_j\exp\left(j_i2^{-m}\right)F\left(A\cap \{\mathbf Z_m=e^{j2^{-m}}\}\right)\\
&=\exp\left(-2^{-m}\right)\int_A Z_{m,i} \d F\\
&\geq \exp\left(-2^{-m+1}\right)\int_A Z_i \d F\\
&= \exp\left(-2^{-m+1}\right)F_i(A),
\end{align*}
where the first equality holds since $X_m$ is independent of $\mathbf Y_m$ given $\mathbf Z''_m$, and the fourth equality holds because $Q\circ X_m^{-1}=F$ and $\mathbf Z_m\circ X_m=\mathbf Z_m''$. Symmetrically,
\begin{equation}\label{eq:upper:bound}
\int_{X_m^{-1}(A)}Y_{m,i}\d Q\leq \exp(2^{-m}) F_i(A).
\end{equation}

Combining (\ref{eq:Qiexact})-(\ref{eq:upper:bound}), we have
$$
Q_i(X_m\in A)\in[\exp(-2^{-m+1})F_i(A),\exp(2^{-m+1})F_i(A)].
$$
Since this holds for any $A\in{\mathcal B}(\mathbb B)$, we conclude that $Q_i\circ X_m^{-1}$ is absolutely continuous with respect to $F_i$, and ${\d Q_i\circ X_m^{-1}}/{\d F_i}\in[\exp(-2^{-m+1}), \exp(2^{-m+1})]$. It is easy to see that $D_{\text{KL}}(Q_i\circ X_m^{-1}\Vert F_i)$ converges to 0 as $m\to\infty$.
 \end{proof}

\end{document}